\documentclass{amsart}

\usepackage{soul}
\usepackage{amsmath,accents}
\usepackage{amsfonts}
\usepackage{amssymb}
\usepackage{mathtools}
\usepackage{latexsym}

\usepackage[shortlabels]{enumitem}
\usepackage{cancel}
\usepackage{cases}
\usepackage{empheq}
\usepackage{multicol}
\usepackage{ulem}
\usepackage{makecell}

\usepackage{wrapfig}

\usepackage{graphicx}

\usepackage{mathrsfs}

\usepackage[T1]{fontenc}
\usepackage[utf8]{inputenc}

\usepackage{verbatim}

\usepackage{url}

\theoremstyle{plain}
\newtheorem{theorem}{Theorem}[section]
\newtheorem{lemma}[theorem]{Lemma}
\newtheorem{corollary}[theorem]{Corollary}
\newtheorem{proposition}[theorem]{Proposition}

\theoremstyle{definition}
\newtheorem{definition}[theorem]{Definition}

\theoremstyle{remark}
\newtheorem{remark}[theorem]{Remark}
\newtheorem{example}[theorem]{Example}

\usepackage{etoolbox}
\AtEndEnvironment{remark}{\hfill$\diamond$}
\AtEndEnvironment{example}{\hfill$\diamond$}

\numberwithin{equation}{section}
\numberwithin{figure}{section}

\usepackage{hyperref}
    \hypersetup{
        colorlinks=true,
        urlcolor=blue,
        citecolor=red
    }

\newcommand{\qqquad}{\qquad\qquad}
\newcommand{\qqqquad}{\qqquad\qqquad}

\newcommand{\vect}[1]{\mathbf{#1}}

\newcommand{\bv}{\vect{v}}
\newcommand{\bw}{\vect{w}}

\newcommand{\bx}{\vect{x}}

\newcommand{\by}{\vect{y}}

\newcommand{\bz}{\vect{z}}

\newcommand{\bnu}{\boldsymbol{\nu}}
\newcommand{\bzeta}{\boldsymbol{\zeta}}
\newcommand{\bphi}{\boldsymbol{\phi}}

\newcommand{\bell}{\boldsymbol{\ell}}

\newcommand{\bE}{\vect{E}}
\newcommand{\bg}{\vect{g}}
\newcommand{\bh}{\vect{h}}

\newcommand{\bbf}{\vect{f}}

\newcommand{\qbar}{{\overline{q}}}

\newcommand{\cA}{\mathcal A}

\newcommand{\cN}{\mathcal N}

\newcommand{\vcN}{\boldsymbol{\mathcal N}}

\newcommand{\R}{\mathbb{R}}
\newcommand{\Rn}{\R^n}

\newcommand{\N}{\mathbb{N}}

\DeclareMathOperator{\supp}{supp}
\DeclareMathOperator{\diver}{div}
\DeclareMathOperator{\grad}{D}

\DeclareMathOperator{\dist}{dist}

\DeclareMathOperator{\NN}{\vect{N}}

\renewcommand{\O}{\Omega}

\renewcommand{\d}{\delta}
\renewcommand{\l}{\lambda}

\newcommand{\s}{\sigma}

\newcommand{\e}{\varepsilon}

\newcommand{\bvphi}{\boldsymbol{\vphi}}

\newcommand{\balpha}{\boldsymbol{\alpha}}

\newcommand{\vphi}{\varphi}

\newcommand{\ov}{\overline}

\newcommand{\abs}[1]{
    \left\lvert#1\right\rvert
}

\newcommand{\set}[1]{
    \left\{#1\right\}
}

\newcommand{\norm}[1]{
    \left\|#1\right\|
}

\newcounter{my_counter}
\usepackage{pifont}

\begin{document}

\title[Nonlocal Boundary Conditions]{Nonlocal Boundary Conditions for Truncated-Fractional Differential
Equations Posed on Bounded Domains}

\author[Mikil Foss]{Mikil Foss}
\address[Mikil Foss]{
Department of Mathematics,
University of Nebraska--Lincoln,
Lincoln, NE 68588-0130, USA}
\email[Mikil Foss]{mikil.foss@unl.edu}

\author[Adam Larios]{Adam Larios}
\address[Adam Larios]{
Department of Mathematics,
University of Nebraska--Lincoln,
Lincoln, NE 68588-0130, USA}
\email[Adam Larios]{alarios@unl.edu}

\author[Michael Pieper]{Michael Pieper}
\address[Michael Pieper]{
Department of Mathematics,
University of Nebraska--Lincoln,
Lincoln, NE 68588-0130, USA}
\email[Michael Pieper (corresponding author)]
{michael.pieper@huskers.unl.edu}
\date{\today}

\begin{abstract}
Nonlocal models have found great success in recent years. In particular, models that replace spatial derivatives with strongly singular integral operators involving finite interaction radii are used in a wide range of applications, including image processing, continuum mechanics, and many other areas. However, from the perspective of mathematical analysis, difficulties arise when considering the appropriate notion of ``boundary'' conditions.  (The nonlocal analogue of the boundary is often called a ``collar'' and is typically not a lower-dimensional set.)  For instance, merely enforcing homogeneous Dirichlet-type constraints on the collar fails to guarantee a number of essential properties used in the analysis and numerical simulation of these models, as we show in the present work. In the local setting, working in a Sobolev space of functions that are zero on the boundary allows one to integrate by parts with no boundary term, approximate by test functions, extend by zero to the full space, and apply the Hardy and Poincar\'e inequalities. We prove that the nonlocal analog of each of these fails if one only requires the functions to be zero on the collar, and that one must move to a smaller, more restrictive space to enjoy these properties. In addition, we introduce a new lifting operator, correct an error in a previously published result on a nonlocal Green's theorem, and establish how the properties listed above relate to existing notions of fractional Sobolev spaces.
\end{abstract}



\maketitle

\section{Introduction}
What are the appropriate boundary conditions for problems involving nonlocal operators? In any modeling application, the answer to this question must be informed by the physical factors that the model aims to capture. In this work, we put aside the modeling concerns and investigate the question mathematically: what are the implications of some of the natural choices of volume constraints to be imposed on nonlocal integro-differential equations?

We will focus on the truncated-fractional operators $\grad^s_\d$ and $\diver^s_\d$, along with the associated function spaces. These operators were introduced in \cite{bellido_non-local_2023} as a truncated version of the Riesz fractional operators studied in \cite{shieh_new_2015, shieh_new_2018, silhavy_fractional_2020} and originally defined in \cite{horvath_composition_1959}. Like the Riesz fractional operators, the truncated fractional differential operators enjoy a number of attractive properties that make them well-suited for modeling and analysis. Unlike the Riesz fractional operators, however, they have only a finite range of interaction and may therefore be applied on bounded domains, with volume constraints enforced only a set of finite measure. This is distinct from works such as \cite{ros-oton_nonlocal_2016}, which study nonlocal problems which enforce constraints on all of $\O^c$.

Most of the current paper focuses on the case $s\in(0,1)$. We note that applications of $\grad^s_\d$ with $s\leq0$ arise in peridynamic models of heat conduction \cite{Bobaru_Duangpanya_2010_IJHMT, Bobaru_Duangpanya_2012_JCP}, where finite-range integrable kernels are a natural modeling choice. Some brief remarks on the case $s\le0$ can be found at the end of the paper.

In the classical setting, homogeneous Dirichlet boundary conditions are imposed by working in the space $H^1_0(\O)$, which consists of all functions with zero trace. When we move to the nonlocal setting, we are forced to consider volume constraints. The integral operators are not sensitive to the values only on the topological boundary since this is a set of measure zero. Within the nonlocal framework, there are several non-equivalent options for interpreting the analog of ``zero Dirichlet data." The primary purpose of this note is to compare and contrast two of these choices.

\subsection{Literature Review and Summary of New Contributions}

Fractional Sobolev spaces have been studied almost as long as their integer-order counterparts, and interest in fractional-order derivatives and spaces has increased rapidly in recent decades (see, for example, \cite{di_nezza_hitchhikers_2012} and the references therein). This is thanks, in part, to a whole host of newly discovered modeling applications of nonlocal and fractional operators. These applications span a range of disciplines, including swarming models \cite{bernoff_nonlocal_2013, mao_nonlocal_2019,mogilner_non-local_1999}, image processing and data classification \cite{Gilboa_Osher_2008_image, kindermann_deblurring_2005,aubert_can_2009,el_bouchairi_nonlocal_2023}, fracture dynamics \cite{lipton_cohesive_2016,bobaru_handbook_2016}, and corrosion modeling \cite{jafarzadeh_peridynamic_2018}, as well as turbulent fluids, flow through porous materials, and fractional diffusion \cite{hilfer2000applications,metzler_random_2000,pozrikidis_fractional_2018}.

There are a variety of approaches for defining fractional-order Sobolev spaces on bounded domains. In this work, we follow an approach that precisely mirrors the standard definition of $H^1(\O)$: define a (fractional-order) weak gradient, use it to define a norm, then work in the Banach space determined by that norm. This requires selecting an operator which behaves similarly to the classical gradient $\nabla$. The papers \cite{silhavy_fractional_2020, shieh_new_2015, shieh_new_2018, bellido_non-local_2023, bellido_nonlocal_2025} provide compelling arguments that there is essentially only one choice for a fractional analog of the gradient, namely, the truncated Riesz fractional gradient $\grad^s_\d$. Related operators have also been studied in \cite{DU_GUNZBURGER_LEHOUCQ_ZHOU_2012_SIAM,mengesha_localization_2015, DuMengesha, Gunzburger2010}. Singular integral operators, including $\grad^s_\d$, have found an especially powerful application in peridynamics \cite{Silling2000} and hyperelasticity \cite{bellido_hyperelasticity_2015}. For additional context, discussion, and references relating to the operator $\grad^s_\d$, we refer to the works \cite{cueto_variational_2023, BeCuFoRa23, kreisbeck_non-constant_2024}.

The main goal of the current work is to investigate Dirichlet-type volume constraints for nonlocal problems on bounded domains.
Several approaches have been proposed in the literature for enforcing volume constraints.  One common choice is to require functions to vanish on a single outer collar $\Gamma_\delta$ (see, e.g., \cite{DuMengesha,DU_GUNZBURGER_LEHOUCQ_ZHOU_2012_SIAM,Silling2000}) which corresponds to the space $Z^{s,\d}(\O)$ studied here.  A second approach imposes constraints on both the inner and the outer collar \cite{Du_Gunzburger_Lehoucq_Zhou_2013_MMMAS,Gunzburger2010}, which can be viewed as a ``double collar'' condition.  While mathematically convenient, such a condition may be unrealistic in applications where only external constraints are available.  A third approach, motivated by approximation theory, defines the relevant spaces as the closure of compactly supported test functions (see, e.g., \cite{bellido_non-local_2023,cueto_variational_2023}), which is the space $\mathring H^{s,\d}(\O)$ studied here.

As noted, the question of appropriate boundary conditions has received considerable attention (see, e.g., \cite{DU_GUNZBURGER_LEHOUCQ_ZHOU_2012_SIAM,Du_Gunzburger_Lehoucq_Zhou_2013_MMMAS,DuMengesha,Tian_Du_2013_SINUM}). The most direct approach of enforcing nonlocal boundary conditions leads to the space $Z^{s,\d}(\O)$, defined in Definition~\ref{def: Zsd}. We show in Corollary~\ref{coro: Zsd} that this space fails to have most of the mathematical properties associated with $H^1_0(\O)$, which we recall in Theorem~\ref{thm: char H10}. For example, there are elements of $Z^{s,\d}(\O)$ which may not be obtained as an $H^{s,\d}(\O)$-limit of test functions supported in $\O$, and there are non-zero elements of $Z^{s,\d}(\O)$ with zero truncated-fractional gradient. On the other hand, Theorem~\ref{thm: char Hsd0} implies that the space $\mathring H^{s,\d}(\O)$, described in Definition~\ref{def: mathringH}, behaves much more analogously to $H^1_0(\O)$.

In this paper, we build substantially on previous results
characterizing fractional Sobolev spaces on bounded domains defined in terms of $\grad^s_\d$. We establish a number of key structural properties and clarify the relationship between different choices of function spaces. 

Part \ref{property: (s,d) zero extension} of Theorem~\ref{thm: char Hsd0} builds on \cite[Lemma 5]{cueto_variational_2023} and \cite[Theorem 3.9(iii)]{bellido_nonlocal_2025}   connecting $\mathring H^{s,\d}(\O)$ with some more traditional fractional Sobolev spaces (see Definition~\ref{def: compVal}). Thus, one consequence of the current paper is a new characterization of the space $\widetilde H^s(\O)$ in terms of a nonlocal integration-by-parts identity, \ref{property: (s,d) IBP}. Corollary~\ref{coro: extend} also corrects and extends the nonlocal integration-by-parts result given in \cite[Lemma 3.6]{BeCuFoRa23}.

The fractional Hardy inequality has been studied for its own sake in several contexts; see, e.g., \cite{dyda_density_2022} and the references therein. One surprising property of $Z^{s,\d}(\O)$ is that its elements need not satisfy a fractional analog of the classical Hardy inequality on bounded domains, even when $\partial\O$ is smooth; as we prove in Corollary \ref{coro: Zsd}.

\section{Preliminaries and Notation}
For convenient reference later on, we record some frequently used hypotheses here.
    \begin{enumerate}[label=(H\arabic*)]
        \item $n$ is a positive integer and $\O\subseteq \R^n$ is a nonempty, open, bounded, connected set. The parameters $s\in(0,1)$ and $\d>0$ are arbitrary. \label{H0}
        \item In addition to \ref{H0}, the boundary of $\O$ is $C^{1,1}$.\label{H1}
    \end{enumerate}

\begin{definition}
    Given a positive integer $n$ and an open set $\O$, write $d_\O: \R^n\to\R$ to denote the function $d_\O(\bx)\coloneqq \dist(\bx,\partial\O)$.
\end{definition}

\begin{definition}
    Given a positive integer $n$, we let $\s_{n-1}$ denote the surface area of the unit sphere in $\R^n$.
\end{definition}

\begin{definition}
    Given a positive integer $n$, a parameter $r>0$, and an open set $\O\subseteq\R^n$ (we denote the complement $\O^c:=\R^n\setminus\O$), we define the following sets
    \begin{align*}
        \Gamma_r&\coloneqq\set{\bx\in\O^c: d_\O(\bx) \le r},\\
        \Gamma_{-r}&\coloneqq\set{\bx\in\O: d_\O(\bx)\le r},\\
        \Gamma_{\pm r}&\coloneqq \Gamma_r\cup\Gamma_{-r}=\set{\bx\in\R^n: d_\O(\bx)\le r},\\
        \O_r&\coloneqq \O+B_r(\mathbf0)=\set{\bx+\bz: \bx\in\O, \bz\in B_r(\mathbf0)}, \text{ and}\\
        \O_{-r}&\coloneqq \O\setminus \Gamma_{-r}=\set{\bx\in\O: d_\O(\bx)>r}.
    \end{align*}
    We sometimes refer to $\Gamma_r$ as an ``outer collar'' and $\Gamma_{-r}$ as an ``inner collar.''
\end{definition}

The next operator is used to explicitly keep track of the domain on which functions are defined.
\begin{definition}
    We write $E_0$ to denote the extension-by-zero operator acting on scalar fields, and $\bE_0$ is the extension-by-zero operator for vector fields. 
\end{definition}
\begin{remark}
    Note that the precise meaning of the extension operators $E_0$ and $\bE_0$ varies based on context, but in all cases, refers to extending a function by zero outside its domain of definition.  More precisely, given a domain $\mathcal U\subset\R^n$ and $u: \mathcal U\to \R$, we define $E_0u:\R^n\to \R$ by
    \[
    E_0 u(\bx) := \begin{cases}
        u(\bx), & \bx\in\mathcal U\\
        0, & \bx\notin\mathcal U.
    \end{cases}
    \]
    So, for instance, if $u: \Omega\to \R$, then $E_0u$ would be zero on $\R^n\setminus \Omega$, but if $u: \Omega_\delta\to \R$, then $E_0u$ would be zero on $\R^n\setminus \Omega_\delta$ .
\end{remark}

This paper is concerned with the truncated-fractional gradient and associated function spaces, which we now define. Previous work on these topics includes \cite{bellido_non-local_2023, BeCuFoRa23, kreisbeck_non-constant_2024, cueto_variational_2023}.
\begin{definition}\label{def: wd and balpha}
    Fix a positive integer $n$ and parameters $s\in[0,1)$, $\d>0$. 
    \begin{enumerate}
        \item We denote by $\bar w_\d: [0,\infty)\to \R$ a smooth, positive, nonincreasing function with support in $[0,\d)$. Further assume that $\bar w_\d$ is constant on a neighborhood of the origin: there are constants $a_0>0$ and $b_0\in(0,1)$ such that $\bar w_\d(r) = a_0$ for all $r\in[0,b_0\d)$.

        Then define $w_\d: \R^n\to\R$ as $w_\d(\bx)\coloneqq \bar w_\d(\abs{\bx})$, so that $w_\d$ is a smooth, radial, positive function supported in $B_\d(\mathbf 0)$. Note that $\norm{w_\d}_{L^\infty(\R^n)} = a_0$.
        
        \item Given $w_\d$ as above, we define an antisymmetric kernel function $\balpha^s_\d: \R^n\times\R^n\to\R^n$ by
        \begin{align*}
            \balpha^s_\d(\bx,\by)\coloneqq \frac{\bx-\by}{\abs{\bx-\by}}\frac{w_\d(\bx-\by)}{\abs{\bx-\by}^{n+s}}.
        \end{align*}
    \end{enumerate}
\end{definition}

\begin{remark}\label{rmk: no c}
    In this work we treat $s$ and $\d$ as fixed, and we are not concerned with asymptotics (such as the behavior as $\d\to0$ or $s\to1$). Hence, we leave the scaling unspecified and do not include an explicit constant $c_{n,s,\d}$ that appears in other works.
\end{remark}

\begin{definition}
    Assume $n$ is a positive integer, the parameters $s\in[0,1)$ and $\d>0$ are fixed, $\O\subseteq\R^n$ is an open set, and $\balpha^s_\d$ is given as in Definition~\ref{def: wd and balpha}. 
    
    \begin{enumerate}
        \item For each $u\in C^\infty(\O_\d)\cap C(\overline{\O_\d})$ and $\bv\in C^\infty(\O_\d;\R^n)\cap C(\overline{\O_\d};\R^n)$, we define the truncated-fractional gradient $\grad^s_\d u$ and truncated-fractional divergence $\diver^s_\d\bv$ by
        \begin{align}
            \grad^s_\d u(\bx)&\coloneqq \int_{B_\d(\bx)} (u(\bx)-u(\by))\balpha^s_\d(\bx,\by)\,d\by, \text{ and }\label{eqt: Dsd}
            \\
            \diver^s_\d \bv(\bx)&\coloneqq \int_{B_\d(\bx)} (\bv(\bx)-\bv(\by))\cdot\balpha^s_\d(\bx,\by)\,d\by.
        \end{align}

        \item For $u\in L^1_{\text{loc}}(\O_\d)$, a function $\bv\in L^1_{\text{loc}}(\O;\R^n)$ is the truncated-fractional gradient of $u$ if and only if 
        \begin{align}\label{def: weak Dsd}
            \forall \bvphi\in C^\infty_c(\O;\R^n), \ \int_\O \bv \cdot \bvphi\,d\bx = -\int_{\O_\d}u\diver^s_\d(\mathbf E_0\bvphi)\,d\bx.
        \end{align}
        In this case, we write $\grad^s_\d u\coloneqq \bv$. It is shown in \cite{cueto_variational_2023} that this definition of $\grad^s_\d u$ coincides with \eqref{eqt: Dsd} for $u\in C^\infty_c(\R^n)$.

        \item Given $p\in[1,\infty)$, we define a norm
        \begin{align}\label{eqt: Hsdnorm}
            \norm{u}_{H^{s,p,\d}(\O)}\coloneqq \left(\norm{u}_{L^p(\O_\d)}^p+\norm{\grad^s_\d u}_{L^p(\O;\R^n)}^p\right)^{1/p}.
        \end{align}

    \end{enumerate}
\end{definition}

\begin{remark}
    An alternative definition of the truncated-fractional gradient is provided in \cite{bellido_non-local_2023}, which proceeds as follows. Given $u\in L^1_{\text{loc}}(\O_\d)$, if there is a sequence $\set{u_k}_{k=1}^\infty\subseteq C^\infty_c(\R^n)$ such that $\norm{u - u_k}_{L^p(\O_\d)}\to0$ as $k\to\infty$ and the sequence $\set{\grad^s_\d u_k}_{k=1}^\infty$ is Cauchy in $L^p(\O;\R^n)$, then we define $\grad^s_\d u\in L^1(\O;\R^n)$ as the unique $L^1(\O;\R^n)$-limit of the Cauchy sequence $\set{\grad^s_\d u_k}_{k=1}^\infty$.

    See \cite{cueto_variational_2023} for a proof that the two definitions coincide whenever $\O\subseteq\R^n$ is either a bounded Lipschitz domain or $\O=\R^n$.
\end{remark}

\begin{definition}
    Given a positive integer $n$, an open set $\O\subseteq\R^n$, some $s\in[0,1)$, $\d>0$, and $1\le p< \infty$, we define the following spaces:
    \begin{align*}
        H^1(\O)&\coloneqq \set{u\in L^2(\O): \nabla u\in L^2(\O;\R^n)},\\
        H^{s,p,\d}(\O)&\coloneqq\set{u\in L^p(\O_\d): \grad^s_\d u\in L^p(\O;\R^n)}, \text{ and }\\
        N^{s,p,\d}(\O)&\coloneqq\set{u\in H^{s,p,\d}(\O): \grad^s_\d u\equiv\mathbf 0 \text{ a.e. in }\O}.
    \end{align*}
    We equip $H^1(\O)$ with the usual norm
    \begin{align*}
        \norm{u}_{H^1(\O)} = \left(\norm{u}_{L^2(\O)}^2 + \norm{\nabla u}_{L^2(\O;\R^n)}^2\right)^{1/2},
    \end{align*}
    while both $H^{s,p,\d}(\O)$ and $N^{s,p,\d}(\O)$ are equipped with the norm \eqref{eqt: Hsdnorm}. 
    
    When $p=2$, we write $H^{s,\d}(\O)$ and $N^{s,\d}(\O)$ instead of $H^{s,2,\d}(\O)$ and $N^{s,2,\d}(\O)$ respectively.
\end{definition}

\begin{remark}\label{rmk: density in Hsd}
    When $1\le p<\infty$ and $\O\subseteq\R^n$ is either a bounded Lipschitz domain or $\O=\R^n$, there is an alternative characterization of $H^{s,p,\d}(\O)$. In this case, \cite[Theorem 1]{cueto_variational_2023} shows that for every $u\in H^{s,p,\d}(\O)$, there is a sequence $\set{u_k}_{k=1}^\infty\subseteq C^\infty_c(\R^n)$ such that $\norm{u-(u_k)|_{\O_{\d}}}_{H^{s,p,\d}(\O)}\to0$ as $k\to\infty$. Thus, the definition of $H^{s,p,\d}(\O)$ given above agrees with the original definition provided in \cite{bellido_non-local_2023}, which described $H^{s,p,\d}(\O)$ as the closure of test functions under the norm \eqref{eqt: Hsdnorm}.
\end{remark}

We now recall some previously established results that will be useful in the sequel.

\begin{theorem}[Consequence of Proposition 3.9, \cite{kreisbeck_non-constant_2024}]\label{thm: psi}
    If \ref{H1} is satisfied and $p\in (1,\frac{2}{1-s})$, a function in $N^{s,p,\d}(\O)$ is completely determined by its mean value in $\O$ and its values on the collar $\Gamma_\d$. The map    
    \begin{align*}
        \Psi^s_\d:N^{s,p,\d}(\O)\to \R\times L^p(\Gamma_\d), \ \ \ \Psi^s_\d(h)\coloneqq \left( \int_\O h \,d\bx, \left. h\right|_{\Gamma_\d}\right)
    \end{align*}
    is a vector space isomorphism and a homeomorphism.
\end{theorem}

As we are concerned with traces, boundary conditions, and integration by parts, we will need a nonlocal version of the normal vector on the boundary. The nonlocal normals, $\vcN$ and $\cN$, were introduced in \cite{BeCuFoRa23} and defined as follows. 
\begin{definition}
    Let $\boldsymbol{\vphi}:\O_\d \to  \Rn$ and $\psi: \O_\d\to \R$ be integrable functions. For each $\bx\in \Gamma_\d\setminus\partial\O$, we denote
\begin{align}
  \vcN\psi(\bx) &\coloneqq\int_{\O} \psi(\by)\balpha^s_\d(\bx,\by)\, d\by \label{eq: nonlocal normal vector}
  \\
  \mathcal{N}\bvphi(\bx)&\coloneqq\int_{\O} \bvphi(\by)\cdot\balpha^s_\d(\bx,\by)\, d\by.\label{eq: nonlocal normal scalar}
\end{align}
\end{definition}
Note that $\vcN$ sends scalar fields to vector fields, while $\cN$ sends vector fields to scalar fields. Using these operators, we can establish a nonlocal analog of integration by parts (also known as Green's identity or the Stokes Formula).

Continuity of $\vcN$ is currently only known when there is extra integrability in a neighborhood of $\partial \O$. One objective of the current paper is to develop some additional theory around the operator $\vcN$ when acting on $Z^{s,\d}(\O)$, building on the bounds given by Lemma~\ref{lem: norm of normal}. A version of this lemma appeared in \cite{BeCuFoRa23}, but there was an error in the second part. We provide a corrected reformulation and proof. As noted in Remark~\ref{rmk: no c}, we do not include the constant $c_{n,s}$.
\begin{lemma}[Corrected form of Lemma 3.6, \cite{BeCuFoRa23}]\label{lem: norm of normal}
    Suppose \ref{H0} holds. Let $1< p_1< \infty$ and $\e\in(0,\d)$ be given. Then for all $u\in L^{p_1}(\O_\d)$,
    \begin{align*}
        \norm{\vcN u}_{L^{p_1}(\Gamma_\d\setminus\Gamma_\e;\R^n)}\le C_1 \norm{u}_{L^{p_1}(\O_\d)}, \quad \text{ where } \ C_1=\frac{a_0}{\e^{n+s}}\abs{\O}^{1/p_1'}\abs{\Gamma_\d}^{1/p_1}.
    \end{align*}
    Additionally, for any $q\in (1,\frac1s)$, there is a $\qbar>\frac1{1-s}$ and a constant
    \begin{align*}
        C_2=C_2(n,s,a_0,p_1,q,\O,\d)<\infty
    \end{align*}
    such that, for all $u:\O_\d\to\R$ satisfying $u|_\O\in L^{p_1}(\O)$ and $u|_{\Gamma_{-\e}}\in L^{\qbar}(\Gamma_{-\e})$,
    \begin{align*}
        \norm{\vcN u}_{L^{q}(\Gamma_\d;\R^n)}&\le \frac{a_0}{\e^{n+s}}\abs{\O}^{1/p_1'}\abs{\Gamma_\d}^{1/q}\norm{u}_{L^{p_1}(\O)} + C_2\norm{u}_{L^{\qbar}(\Gamma_{-\e})}.
    \end{align*}
\end{lemma}
\begin{proof}
    The first part is proven in \cite{BeCuFoRa23}. The steps of the proof of the second part are also essentially correct, but they do not fully establish the Lemma as stated in \cite{BeCuFoRa23}. Thus, we only show how to use the ingredients of \cite{BeCuFoRa23} to form a complete proof of the corrected statement given above.

    Let $1<p_1<\infty$, $1<q<\frac1s$, and $0<\e<\d$ be given. Following \cite{BeCuFoRa23}, we define a linear operator $\vcN^\e$ as follows: for each $u\in L^1(\Gamma_{-\e})$, the function $\vcN^\e u: \Gamma_\d\to \R^n$ is given by
    \begin{align*}
        \forall\bx\in\Gamma_\d, \quad  \vcN^\e u(\bx) \coloneqq \int_{\O\cap B_\e(\bx)}u(\by)\balpha^s_\d(\bx,\by)\,d\by.
    \end{align*}
    Then for any $u\in L^1(\O_\d)$ with $u|_{\O}\in L^{p_1}(\O)$, we have
\begin{align}
\label{eqt: boundingNormal}
    \norm{\vcN u}_{L^q(\Gamma_\d;\R^n)} 
    =&
    \norm{ \int_{\O\cap B_\e(\cdot)^c} u(\by)\balpha^s_\d(\cdot,\by)\,d\by
        +\vcN^\e u}_{L^q(\Gamma_\d;\R^n)}\\
    \le& 
    \left(\int_{\Gamma_\d} \abs{\int_{\O\cap B_\e(\bx)^c} u(\by)\balpha^s_\d(\bx,\by)\,d\by}^q\! d\bx\right)^{1/q}
        \!+ \norm{\vcN^\e u}_{L^q(\Gamma_\d;\R^n)}.\notag
\end{align}
We handle these two pieces separately. For the first term, apply H\"older's inequality with the conjugates $p_1,p_1'$ on the domain $\O\cap B_\e(\bx)^c$ and use the fact that $\abs{\bx-\by}>\e$ for $\by\in \O\cap B_\e(\bx)^c$ to compute
\begin{align*}
    &\int_{\Gamma_\d}\abs{\int_{\O\cap B_\e(\bx)^c}
        u(\by)\balpha^s_\d(\bx,\by)\,d\by}^qd\bx\\
    &\qqqquad\le 
    \int_{\Gamma_\d}\norm{u}_{L^{p_1}(\O\cap B_\e(\bx)^c)}^q
        \norm{\balpha^s_\d(\bx,\cdot)}_{L^{p_1'}(\O\cap B_\e(\bx)^c)}^qd\bx\\
    &\qqqquad\le 
    a_0^q\norm{u}_{L^{p_1}(\O)}^q
        \int_{\Gamma_\d} \abs{\int_{\O\cap B_\e(\bx)^c}
            \abs{\bx-\by}^{-p_1'(n+s)}\,d\by}^{q/p_1'}d\bx\\
    &\qqqquad\le
    \left(\frac{a_0}{\e^{n+s}}\right)^q\norm{u}^q_{L^{p_1}(\O)}
        \abs{\O}^{q/p_1'}\abs{\Gamma_\d}.
\end{align*}
    This completes the bound for the first term on the right-hand side of \eqref{eqt: boundingNormal}. All that remains is to bound $\vcN^\e u$. To do so, we start by recalling two inequalities correctly proven in \cite{BeCuFoRa23}.

    First, if $r\in(1,1/s)$, then for all $u\in L^\infty(\Gamma_{-\e})$,
    \begin{align}\label{eqt: rBound}
        \norm{\vcN^\e u}_{L^r(\Gamma_\d;\R^n)}\le \frac{\s_{n-1}a_0}{s}\norm{d_\O^{-s}}_{L^r(\Gamma_\d)}\norm{u}_{L^\infty(\Gamma_{-\e})}.
    \end{align}
    Second, if $p>\frac{1}{1-s}$, then for all $u\in L^{p}(\Gamma_{-\e})$, 
    \begin{align}\label{eqt: pBound}
        \norm{\vcN^\e u}_{L^1(\Gamma_\d;\R^n)}\le \frac{\s_{n-1} a_0}{s}\norm{d_\O^{-s}}_{L^{p'}(\O)}\norm{u}_{L^{p}(\Gamma_{-\e})}.
    \end{align}
    The assumptions that $r<1/s$ and $p>\frac1{1-s}$ ensure that both $sr,sp'<1$. Hence, Proposition~\ref{prop: distIntegrable} implies that the right-hand sides of \eqref{eqt: rBound} and \eqref{eqt: pBound} are finite.

    After establishing \eqref{eqt: rBound} and \eqref{eqt: pBound}, the authors of \cite{BeCuFoRa23} applied the Riesz-Thorin Interpolation Theorem to the operator $\vcN^\e: L^{p}(\Gamma_{-\e})+L^\infty(\Gamma_{-\e})\to L^1(\Gamma_\d;\R^n)+L^r(\Gamma_\d;\R^n)$. The conclusion is that, for any $p>\frac1{1-s}$ and any $1<r<\frac1s$, the operator $\vcN^\e: L^{p_\theta}(\Gamma_{-\e})\to L^{r_\theta}(\Gamma_\d;\R^n)$ is bounded for each $\theta\in[0,1]$, where
    \begin{align}\label{eqt: interp}
        p_\theta = \frac{p}{1-\theta}, \ \text{ and } \ \frac{1}{r_\theta} = 1-\theta + \frac{\theta}{r}.
    \end{align}
    Thus, it suffices to find $p>\frac1{1-s}$ and $r\in(1,\frac1s)$ such that $q = r_{\theta_0}$ for some $\theta_0\in[0,1]$.

    There are many possibilities for $p$ and $r$, and many ways to find them. The goal here is to just find a single pair that works using the parameters given. To this end, set
    \begin{align*}
        p\coloneqq \frac1{1-s}+\e, \ \text{ and } \ r\coloneqq\frac{q+\frac1s}{2}.
    \end{align*}
    Then $p>\frac1{1-s}$ and $1<q<r<\frac1s$, meaning that $p$ and $r$ satisfy the requirements above and $\vcN^\e: L^{p_\theta}(\Gamma_{-\e})\to L^{r_\theta}(\Gamma_\d;\R^n)$ is bounded for any $\theta\in[0,1]$. Setting
    \begin{align}\label{eqt: theta}
        \theta_0\coloneqq \frac{(sq+1)(q-1)}{q(1+sq-2s)},
    \end{align}
    we see that $\theta_0\in(0,1)$ since $1<q<\frac1s$. This value of $\theta_0$ is selected so that $q = r_{\theta_0} = 1-\theta_\theta+\frac{\theta_0}{r}$. Given $\theta_0$ and $p$, we can then compute $p_{\theta_0}$:
    \begin{align}\label{eqt: qbar}
        \qbar\coloneqq p_{\theta_0} = \frac{p}{1-\theta_0} = \frac{1+\e(1-s)}{(1-\theta_0)(1-s)}.
    \end{align}
    We conclude that $\vcN^\e: L^{\qbar}(\Gamma_{-\e})\to L^q(\Gamma_\d;\R^n)$ is a bounded linear operator whose norm satisfies
    \begin{align*}
        \norm{\vcN^\e}\le \frac{\s_{n-1}a_0}{s}\norm{d_\O^{-s}}_{L^{p'}(\O)}^{1-\theta_0} \norm{d_\O^{-s}}_{L^r(\Gamma_\d)}^{\theta_0}\eqqcolon C_2.
    \end{align*}
\end{proof}

To simplify notation, we reserve the symbol $\qbar$ to denote an amount of additional integrability required on $u$ to obtain an $L^q$ bound on $\vcN u$.
\begin{definition}\label{def: qbar}
    Under the assumption \ref{H1}, let $q\in(1,\min\set{\frac1s,2})$ and $\e\in(0,\d)$ be given. Then define $\qbar$ by \eqref{eqt: qbar}, where $\theta_0$ is given by \eqref{eqt: theta}.
\end{definition}

Here is one of several ways of stating the nonlocal integration by parts rule, using the notation $\qbar$ just defined.

\begin{theorem}[Theorem 3.10, \cite{BeCuFoRa23}]\label{thm: by parts}
	Assume \ref{H0} holds, and $u\in H^{s,\d}(\O)$ is a given function such that there exist $q_1\in(1,1/s)$ and $\e>0$ with $u|_{\Gamma_{\pm\e}}\in L^{\overline{q_1}}(\Gamma_{\pm\e})$. Also suppose that $\bv\in H^{s,\d}(\O;\Rn)$ with $\bv|_{\Gamma_{\pm\e}}\in L^{\overline{q_2}}(\Gamma_{\pm\e})$ for some $q_2\in(1,1/s)$. Then the following equality holds
\begin{equation} \label{eq: nonlocal intergration by parts Hsp}
    \int_{\O}\left[\grad^s_\d u(\bx)\cdot\bv(\bx)
        +u(\bx)\diver^s_\d \bv(\bx)\right]d\bx
    =\int_{\Gamma_\d}\left[\bv(\bx)\cdot\vcN u(\bx)
        +u(\bx)\cN\bv(\bx)\right]d\bx.
\end{equation}
\end{theorem}

\begin{remark}
    In this paper, we will mostly be concerned with functions $u\in H^{s,\d}(\O)$ which are zero on the outer collar $\Gamma_\d$. In this case, the nonlocal Green's identity takes the form
\begin{align}\label{eqt: greens}
    \forall \bvphi\in C^\infty_c(\R^n;\R^n),\qquad
    \int_\O\left(\grad^s_\d u\cdot \bvphi+u\diver^s_\d\bvphi\right)d\bx
    =\int_{\Gamma_\d} \bvphi\cdot\vcN u\, d\bx.
\end{align}
\end{remark}

We conclude by recalling a standard fact about Lipschitz domains. For context and applications, see for example Remarks 6.6 and 11.10 of \cite{kufner_weighted_1985}. The case where $\O$ is assumed to be $C^2$ is handled inside the proof of \cite[Lemma 5.11]{souplet_gradient_2002}. For the sake of completeness, we include a proof of the Lipschitz case in Appendix~\ref{app: distIntegrable}.
\begin{proposition}\label{prop: distIntegrable}
    If \ref{H0} holds and $\O$ is of class $C^{0,1}$, then for any $\kappa>-1$, $d_\O^{\kappa}$ is in $L^1(\O)$.
\end{proposition}

\section{A Review of the Classical Setting}
The space $H^1_0(\O)$ plays a central role in the study of minimization problems in the calculus of variations and PDEs with Dirichlet boundary conditions. The next classical theorem summarizes several well-known key properties of the space $H^1_0(\O)$, most of which can be found in standard textbooks such as \cite{brezis_functional_2011, Girault_Raviart_1986, grisvard_elliptic_1985} and the references therein. Additional references can be found in the proof sketch below. We recall these results here, as one of the main goals of the paper is to provide a nonlocal analogue of Theorem~\ref{thm: char H10} in the truncated-fractional setting, as given by Theorem ~\ref{thm: char Hsd0} and Corollary \ref{coro: Zsd} below.

\begin{theorem}\label{thm: char H10}
    Let $\O\subseteq\R^n$ be an open, bounded domain with $C^{1,1}$ boundary. Then the following are equivalent characterizations of the space $H^1_0(\O)$.
    \begin{enumerate}[label=(C\arabic*)]
    \item $H^1_0(\O)$ is the closure of $C^\infty_c(\O)$ with respect to the $H^1(\O)$-norm. \label{property: closure of test functions}

    \item $H^1_0(\O) = \set{u|_{\O} : u\in H^1(\R^n), \ \supp u\subseteq \overline\O}$ and $u\in H^1_0(\O)$ if and only $E_0u\in H^1(\R^n)$. \label{property: zero extension}
    
    \item $H^1_0(\O)$ is the space of functions $u\in H^1(\O)$ such that $u/d_\O \in L^2(\O)$.
    \label{property: hardy}
    
    \item $H^1_0(\O)$ is the closed subspace of $H^1(\O)$ where Green's identity holds with no boundary term: given $u\in H^1(\O)$, $u$ is in $H^1_0(\O)$ if and only if
    \begin{align*}
        \int_\O u\diver\bv + \nabla u\cdot \bv \, d\bx = 0, \ \ \forall \bv\in C^\infty(\O;\R^n)\cap C(\overline{\O};\R^n).
    \end{align*}\label{property: IBP}

    \item $H^1_0(\O)$ is the kernel of the trace operator $\gamma_0:H^1(\O)\to H^{1/2}(\partial \O)$. \label{property: trace zero}
\end{enumerate}
Moreover, $H^1_0(\O)$ has these additional properties:
\begin{enumerate}[label=(C5\roman{enumi})]
    \item If $u\in H^1_0(\O)$ and $\nabla u = 0$ in $\O$, then $u=0$ in $\O$. \label{property: injective}
    \item There exists $C>0$ such that, for all $u\in H^1_0(\O)$, we have $\norm{u}_{L^2(\O)}\le C\norm{\nabla u}_{L^2(\O;\R^n)}$. \label{property: poincare}
\end{enumerate}
\end{theorem}

\begin{proof}
We take \ref{property: closure of test functions} to be the definition of $H^1_0(\O)$. A proof that \ref{property: trace zero} is equivalent can be found in Theorem 1.5 of \cite{Girault_Raviart_1986}. 
Statement \ref{property: zero extension} is proven in Proposition 9.18 of \cite{brezis_functional_2011}, and \ref{property: hardy} is proven in \cite{kadlec_characterization_1966} as well as \cite[Theorem V.3.4]{edmunds_spectral_1987} and \cite[Example 9.12]{kufner_weighted_1985}.

Statement \ref{property: IBP} on Lipschitz domains follows from the generalized Stokes formula (Theorem 1.2 in Chapter 1 of \cite{Temam_2001_Th_Num}), the density of $C^\infty(\O;\R^n)\cap C(\overline{\O};\R^n)$ in $\mathbf E(\O)\coloneqq \set{\bv\in L^2(\O;\R^n): \diver\bv \in L^2(\O)}$, and the surjectivity of the bounded linear operator $\gamma_{\bnu}: \mathbf E(\O)\to H^{-1/2}(\partial\O)$ which measures the normal component of $\bv\in \mathbf E(\O)$ on the boundary. See Theorem 1.2 and Remark 1.2 of Chapter 1, \cite{Temam_2001_Th_Num}.

If $\nabla u = 0$, then $u$ is constant on each connected component of $\O$. Since $u$ has zero trace, $u=0$. The last property is the Poincar\'e inequality, which can be found, for example, in Corollary 9.19 of \cite{brezis_functional_2011}.
\end{proof}

\begin{remark}
    In the previous theorem, the assumption that the boundary is $C^{1,1}$ is stronger than necessary. We state it this way for more direct comparison with the fractional cases studied in later sections of this paper.
\end{remark}

\section{Transitioning to the Nonlocal Setting}
When constructing the mathematical theory of truncated-fractional PDEs, we must confront the following questions: which function space should play the role of $H^1_0(\O)$? What is the correct nonlocal analog of the classical, local, zero-Dirichlet boundary conditions? Which of the properties from Theorem~\ref{thm: char H10} are essential? Which should serve as the basis for generalizing to the new, fractional setting?

Approximation by ``nice" functions is a convenient and sometimes essential tool in both theoretical analysis and numerical simulation. Seeing as Property \ref{property: closure of test functions} is taken as the definition of $H^1_0(\O)$, it is natural to work with the closure $C^\infty_c(\O)$ in the $H^{s,\d}(\O)$-norm rather than the classical $H^1(\O)$-norm. This leads to the first function space under consideration, which we denote by $\mathring H^{s,\d}(\O)$.

\begin{definition}\label{def: mathringH}
    The space $\mathring H^{s,\d}(\O)$ is the closed subspace of $H^{s,\d}(\O)$ defined as the closure of smooth functions which are compactly supported in $\O$:
    \begin{align*}
        \mathring H^{s,\d}(\O)\coloneqq\overline{\set{\vphi\in C^\infty_c(\O_\d): \supp\vphi\subseteq\O}}^{H^{s,\d}(\O)}\subseteq H^{s,\d}(\O).
    \end{align*}
\end{definition}

Alternatively, one could consider the closure with respect to the $H^{s,\d}$-norm on full space.
\begin{definition}\label{def: widetildeH}
    The space $\widetilde H^{s,\d}(\O)$ is the closed subspace of $H^{s,\d}(\R^n)$ defined as the closure of smooth functions which are compactly supported in $\O$:
    \begin{align*}
        \widetilde H^{s,\d}(\O)\coloneqq\overline{\set{\vphi\in C^\infty_c(\R^n): \supp\vphi\subseteq\O}}^{H^{s,\d}(\R^n)}\subseteq H^{s,\d}(\R^n).
    \end{align*}
\end{definition}
 We focus on $\mathring H^{s,\d}(\O)$, but Proposition~\ref{prop: RingToTildeHomeo} shows that $\mathring H^{s,\d}(\O)$ is canonically isomorphic to $\widetilde H^{s,\d}(\O)$.

Theorem~\ref{thm: char Hsd0} demonstrates that $\mathring H^{s,\d}(\O)$ enjoys many of the same properties appearing in Theorem~\ref{thm: char H10} in the nonlocal setting. One could argue, however, that this space is ``too small." Indeed, in many situations the key property of $H^1_0(\O)$ is not that its elements are approximated by smooth functions, but rather that its elements have zero trace. From this perspective, taking the definition to be the closure of test functions is a mere mathematical convenience: the ``true" motivation for defining the space is to work with Dirichlet-type boundary conditions. From this perspective, the key property that we should aim to generalize in the nonlocal setting is \ref{property: trace zero}, not \ref{property: closure of test functions}. This leads us to the second function space of this paper, which will be discussed in more detail in Section~\ref{sec: zeroCollar}.

\begin{definition}\label{def: Zsd}
    Given a bounded domain $\O$ and $\delta>0$, we define $Z^{s,\d}(\O) = \set{u\in H^{s,\d}(\O): u|_{\Gamma_\delta} = 0}$.
\end{definition}
Equivalently, $Z^{s,\d}(\O)$ is the kernel of the nonlocal boundary-value operator $\gamma_\d: H^{s,\d}(\O)\to L^2(\Gamma_\d)$, defined by $\gamma_\d u \coloneqq u|_{\Gamma_\d}$, which plays an analogous role to the classical trace $\gamma_0$.

\subsection{Lifts and the Nonlocal Normal}
In the local setting on a bounded domain, one can define a trace operator $\gamma_0: H^1(\O)\to H^{1/2}(\partial\O)$ such that $\gamma_0 u = \left. u\right|_{\partial \O}$ for all $u\in C^\infty(\O)\cap C(\overline\O)$. This provides a way of making sense of boundary values, which are otherwise undefined since the topological boundary $\partial \O\subseteq \R^n$ is a set of measure zero. The map $\gamma_0: H^1(\O)\to H^{1/2}(\partial\O)$ is continuous, linear, and surjective. In particular, it admits a continuous one-sided inverse, the lifting operator $\ell: H^{1/2}(\partial\O) \to H^1(\O)$, such that $\gamma_0\ell$ is the identity in $H^{1/2}(\partial\O)$. Thinking of $\Gamma_\d$ as the nonlocal analog of the boundary $\partial\O$, we define a \textit{nonlocal} lifting operator as follows.

\begin{definition}\label{def: lift}
    Assume \ref{H1}. For $p\in (1,\frac{2}{1-s})$, define the lifting operator
    \begin{align*}
        \ell^s_\d: L^p(\Gamma_\d)\to N^{s,p,\d}(\O)\subseteq H^{s,p,\d}(\O), \ \ \ \ell^s_\d(g)\coloneqq \left(\Psi^s_\d\right)^{-1}\left( 0, g\right),
    \end{align*}
    where $\Psi^s_\d$ is the isomorphism given in Theorem \ref{thm: psi}.

    For vector-valued functions $\bbf\in L^p(\Gamma_\d;\R^n)$, define $\bell^s_\d(\bbf)\in H^{s,p,\d}(\O;\R^n)$ by applying $\ell^s_\d$ to each component of $\bv$: $\left( \bell^s_\d(\bbf)\right)_j\coloneqq \ell^s_\d(f_j)$, $1\le j\le n$.
\end{definition}

\begin{remark}\label{rmk: lift}
    The definition of the lift $\ell^s_\d$ ensures that, for $p\in(1,\frac{2}{1-s})$ and $g\in L^p(\Gamma_\d)$, $\ell^s_\d(g) = g$ a.e. in $\Gamma_\d$ and $\grad^s_\d\left(\ell^s_\d(g)\right) = 0$ a.e. in $\O$. For vector fields, this means that each component of $\bell^s_\d(\bbf)$ has zero truncated-fractional gradient. Hence, all truncated-fractional partial derivatives $\grad^s_{\d,j}$ of all components $f_k$ are zero. In particular, $\diver^s_\d \left(\bell^s_\d(\bbf) \right)= 0$ in $\O$ for any $\bbf\in L^p(\Gamma_\d;\R^n)$. Additionally, Theorem~\ref{thm: psi} implies the lift is a bounded linear operator into $N^{s,p,\d}(\O)$, which can be equipped with either the $H^{s,\d}(\O)$ norm or the $L^p(\O_\d)$ norm.
\end{remark}

Using this lift, we define an operator $\NN$, thought of as a distributional nonlocal normal. We show later that, for sufficiently nice functions, this operator agrees with the nonlocal normal $\vcN$ introduced by \cite{BeCuFoRa23} and defined in \eqref{eq: nonlocal normal vector}.
\begin{definition}\label{def: NN}
    Suppose \ref{H1} holds. Given $u\in Z^{s,\d}(\O)$, define a linear functional $\NN u$ acting on $L^2(\Gamma_\d;\R^n)$ by
\begin{align*}
    \forall \bbf\in L^2(\Gamma_\d;\R^n),\qquad
    \langle \NN u,\bbf\rangle_{L^2(\Gamma_\d;\R^n)^*,L^2(\Gamma_\d;\R^n)}
    \coloneqq
    \int_\O\bell^s_\d\bbf(\bx)\cdot\grad^s_\d u(\bx)\,d\bx .
\end{align*}
    Similarly, for all $\bv\in Z^{s,\d}(\O;\R^n)$, define $\NN\cdot \bv\in L^2(\Gamma_\d)^*$ by
\begin{align*}
    \forall g\in L^2(\Gamma_\d),\qquad
    \langle\NN\cdot\bv,g\rangle_{L^2(\Gamma_\d)^*,L^2(\Gamma_\d)}
    \coloneqq
    \int_\O\ell^s_\d g(\bx)\diver^s_\d\bv(\bx)\,d\bx .
\end{align*}
    These functionals are continuous since the nonlocal lift is continuous and $u$ and $\bv$ have $L^2$ nonlocal derivatives. Via the Riesz Representation Theorem, we identify $\NN u$ with the unique function $\bw\in L^2(\Gamma_\d;\R^n)$ such that
    \[
    \forall \bbf\in L^2(\O;\R^n), \ \langle \NN u,\bbf\rangle_{L^2(\Gamma_\d;\R^n)^*,L^2(\Gamma_\d;\R^n)} = \int_{\Gamma_\d} \bbf\cdot \bw d\bx.
    \]
    Similarly, for each $\bv\in Z^{s,\d}(\O;\R^n)$, we identify $\NN\cdot \bv$ as an element of $L^2(\Gamma_\d)$.
\end{definition}

The next lemma shows that the distributional nonlocal normal $\NN:Z^{s,\d}(\O)\to L^2(\Gamma_\d;\R^n)$ is a bounded linear operator, which is still unknown for the pointwise-defined nonlocal normal $\vcN: Z^{s,\d}(\O)\to L^2(\Gamma_\d;\R^n)$.

\begin{lemma}\label{lem: NN is bdd}
    Assume \ref{H1}. The operators $\NN: Z^{s,\d}(\O)\to L^2(\Gamma_\d;\R^n)$ and $\NN\cdot : Z^{s,\d}(\O;\R^n)\to L^2(\Gamma_\d)$ are both continuous, with $\norm{\NN}\le \norm{\bell^s_\d}$ and $\norm{\NN\cdot}\le \norm{\ell^s_\d}$. Here, $\norm{\bell^s_\d}$ is the operator norm of $\bell^s_\d: L^2(\Gamma_\d;\R^n)\to L^2(\O_\d;\R^n)$ and analogously for $\norm{\ell^s_\d}$.
\end{lemma}
\begin{proof}
    We show that the continuity of $\NN$ follows from the continuity of $\bell^s_\d$.
    
    Fix any $u\in H^{s,\d}(\O)$. We have
    \begin{align*}
        \norm{\NN u}_{L^2(\Gamma_\d)} = \sup_{\substack{\bbf\in L^2(\Gamma_\d;\R^n)\\ \norm{\bbf}_{L^2(\Gamma_\d;\R^n)}\le 1}}\abs{\int_{\Gamma_\d} \bbf \cdot \NN u d\bx}= \sup_{\substack{\bbf\in L^2(\Gamma_\d;\R^n)\\ \norm{\bbf}_{L^2(\Gamma_\d;\R^n)}\le 1}}\abs{\int_\O \bell^s_\d \bbf\cdot \grad^s_\d u d\bx}.
    \end{align*}
    Applying the Cauchy-Schwarz inequality and using the fact that $\bell^s_\d$ is a bounded linear operator, 
    \begin{align*}
        \norm{\NN u}_{L^2(\Gamma_\d;\R^n)}&\le \sup_{\substack{\bbf\in L^2(\Gamma_\d;\R^n)\\ \norm{\bbf}_{L^2(\Gamma_\d;\R^n)}\le 1}} \norm{\bell^s_\d \bbf}_{L^2(\O;\R^n)}\norm{\grad^s_\d u}_{L^2(\O;\R^n)}\\
        &\le  \sup_{\substack{\bbf\in L^2(\Gamma_\d;\R^n)\\ \norm{\bbf}_{L^2(\Gamma_\d;\R^n)}\le 1}} \norm{\bell^s_\d} \norm{\bbf}_{L^2(\Gamma_\d;\R^n)}\norm{u}_{H^{s,\d}(\O)}\\
        &= \norm{\bell^s_\d}\norm{ u}_{H^{s,\d}(\O)}
    \end{align*}
    Therefore, the operator norm of $\NN$ is bounded above by the operator norm of $\bell^s_\d$. A similar argument shows that the operator norm of $\NN\cdot$ is controlled by the operator norm of $\ell^s_\d$.
\end{proof}

We will see later that whenever $u\in Z^{s,\d}(\O)$ and $u/d_\O^s\in L^2(\O_\d)$, Proposition~\ref{prop: bound normal in terms of hardy} gives $L^q$ bounds on $\vcN u$ for $q\in(1,2)$. Alternatively, when $u\in Z^{s,\d}(\O)$ satisfies some extra integrability conditions in a neighborhood of $\partial\O$, Lemma~\ref{lem: norm of normal} gives $L^q$ bounds on $\vcN u$ for $q<\frac1s$. Proposition~\ref{prop: characterize N} allows us to upgrade these estimates to $q=2$, even for $s$ close to 1.

\begin{proposition}\label{prop: characterize N}
    If \ref{H1} is satisfied, $u\in Z^{s,\d}(\O)$, $\norm{\vcN u}_{L^q(\Gamma_\d;\R^n)}<\infty$ for some $q\in\left(\frac{2}{1+s},2\right]$, and $u$ satisfies \eqref{eqt: greens}, then $\NN u = \vcN u$ as elements of $L^2(\Gamma_\d;\R^n)$. As a result, $\norm{\vcN u}_{L^2(\Gamma_\d;\R^n)}\le \norm{\NN}\norm{u}_{H^{s,\d}(\O)}<\infty$.

\end{proposition}

\begin{proof}

Let $u\in Z^{s,\d}(\O)$ and $\frac{2}{1+s}<q\le 2$ be given as in the statement of the proposition. 
To show that $\NN u = \vcN u$ as elements of $L^2(\Gamma_\d)$, it suffices to show that $\NN u(\bx) = \vcN u(\bx)$ for almost all $\bx\in \Gamma_\d$, since $\NN u$ is already known to be in $L^2(\Gamma_\d;\R^n)$. By assumption, $\vcN u\in L^q(\Gamma_\d;\R^n)$. Thus, to show that $\NN u$ and $\vcN u$ agree pointwise almost everywhere, we will show that 
\[
\langle \NN u,\bbf\rangle_{L^q(\Gamma_\d;\R^n),L^{q'}(\Gamma_\d;\R^n)} = \langle \vcN u,\bbf\rangle_{L^q(\Gamma_\d;\R^n),L^{q'}(\Gamma_\d;\R^n)}
\]
for all $\bbf\in L^{q'}(\Gamma_\d;\R^n)$. We first establish this equality for $\bbf\in C_c^\infty(\mathop{\mathrm{Int}}\Gamma_\d;\R^n)$ and later extend by density. To this end, fix $\bbf\in C_c^\infty(\mathop{\mathrm{Int}}\Gamma_\d;\R^n)$ and $\e_0>0$.

Applying Lemma~\ref{lem: uniform Lq' bd} to $\bbf$, we obtain a sequence $\set{\bg_k}_{k=1}^\infty\subseteq C^\infty_c(\R^n;\R^n)$ and a constant $C<\infty$, independent of $k\in\N$, such that
\begin{align*}
    \forall k\in\N, \qquad \norm{\bell^s_\d\bbf - \bg_k}_{H^{s,\d}(\O;\R^n)}<1/k \ \text{ and } \ \norm{\bg_k}_{L^{q'}(\Gamma_\d;\R^n)}\le C\norm{\bbf}_{L^{q'}(\Gamma_\d;\R^n)}.
\end{align*}
Put $M\coloneqq C\norm{\bbf}_{L^{q'}(\Gamma_\d;\R^n)}$, and note that $M$ depends only on $\O$, $q$, $\d$, and $\bbf$.

 For $\e\in(0,\d)$, $\Gamma_\e\subseteq\Gamma_\d$, so $\norm{\vcN u}_{L^q(\Gamma_\e;\R^n)}<\infty$. Thanks to the absolute continuity of the Lebesgue integral, we can make $\norm{\vcN u}_{L^{q}(\Gamma_\e;\R^n)}$ arbitrarily small, by selecting $\e>0$ sufficiently small. This, along with the fact that $\supp(\bbf)\Subset \mathop{\mathrm{Int}}\Gamma_\d$, we deduce that there is an $\e_1\in (0,\d)$ such that
\begin{align}\label{eqt: p bound on normal}
    \norm{\vcN u}_{L^{q}(\Gamma_{\e_1};\R^n)} < \frac{\e_0}{4M}, \text{ and } \Gamma_{\e_1} \cap \supp(\bbf)  = \emptyset.
\end{align}
Lemma~\ref{lem: norm of normal} then implies $\norm{\vcN u}_{L^2(\Gamma_\d\setminus\Gamma_{\e_1};\R^n)} <\infty$, where the upper bound depends only on $n,s,a_0,\e_1,\O,\d$, and $\norm{u}_{L^2(\O)}$. Hence,
\begin{align*}
    M_1\coloneqq 1+\max\set{\norm{u}_{L^2(\O)}\norm{\diver^s_\d},\norm{\grad^s_\d u}_{L^2(\O;\R^n)}, \norm{\vcN u}_{L^2(\Gamma_\d\setminus\Gamma_{\e_1};\R^n)} }<\infty.
\end{align*}
As $\bg_k\to \bell^s_\d\bbf$ in $H^{s,\d}(\O;\R^n)$, for sufficiently large $k\in\N$, we find
\begin{align*}
    \norm{\bell^s_\d\bbf - \bg_k}_{H^{s,\d}(\O;\R^n)}<\e_0/(4M_1).
\end{align*}

Let $E$ denote the difference between the action of $\NN u$ and $\vcN u$ on $\bbf$:
\[
    E \coloneqq\abs{\int_{\Gamma_\d}\bbf\cdot\vcN u\,d\bx
        -\int_{\Gamma_\d}\bbf\cdot\NN u\,d\bx}.
\]
Recalling the definition of $\NN u$ and that $\bbf$ is zero on $\Gamma_{\e_1}$, we obtain
\[
    E=\abs{\int_{\O} \bell^s_\d \bbf\cdot \grad^s_\d u\,d\bx
        -\int_{\Gamma_\d\setminus\Gamma_{\e_1}} \bbf\cdot \vcN u\,d\bx}.
\]
After adding and subtracting,
\begin{subequations}
\begin{align}
    E&\le\left|
        \int_\O(\bell^s_\d\bbf-\bg_k)\cdot\grad^s_\d u\,d\bx\right|
    \label{eqt: 1}\\
    &\qqquad+
    \left|
        \int_\O u\diver^s_\d\bg_k\,d\bx\right|
    \label{eqt: 2}\\
    &\qqquad+
    \left|
        \int_\O\bg_k\cdot\grad^s_\d u\,d\bx
        +\int_\O u\diver^s_\d\bg_k\,d\bx
        -\int_{\Gamma_\d\setminus\Gamma_{\e_1}}\bbf\cdot\vcN u\,d\bx
    \right|.
    \label{eqt: 3}
\end{align}
\end{subequations}
We now handle each term on the right-hand side of the inequality separately.

For the right-hand side of \eqref{eqt: 1}, the Cauchy-Schwarz inequality gives
\[
    \left|\int_\O \bell^s_\d\bbf\cdot\grad^s_\d u\,d\bx
        -\int_\O\bg_k\cdot\grad^s_\d u\,d\bx\right| 
    \le\norm{\bell^s_\d\bbf - \bg_k}_{L^2(\O;\R^n)}
        \norm{\grad^s_\d u}_{L^2(\O;\R^n)}.
\]
For sufficiently large $k\in\N$, we have $\norm{\bell^s_\d\bbf - \bg_k}_{L^2(\O;\R^n)} < \frac{\e_0}{4M_1}$, where the definition of $M_1$ ensures $\norm{\grad^s_\d u}_{L^2(\O;\R^n)}\le M_1$. Thus, $\norm{\bell^s_\d\bbf - \bg_k}_{L^2(\O;\R^n)}\norm{\grad^s_\d u}_{L^2(\O;\R^n)}< \frac{\e_0}{4}$.

Turning to \eqref{eqt: 2}, by the definition of the lift, we have $\diver^s_\d(\bell^s_\d\bbf) = 0$ in $\O$. Thus,
\begin{align*}
    \abs{\int_\O u\diver^s_\d\bg_k\, d\bx }
    &=\abs{\int_\O u\diver^s_\d (\bell^s_\d\bbf)\,d\bx
        -\int_\O u\diver^s_\d\bg_k\,d\bx}\\
    &\le\norm{u}_{L^2(\O)}\norm{\diver^s_\d(\bell^s_\d\bbf-\bg_k)}_{L^2(\O)}\\
    &\le\norm{u}_{L^2(\O)}\norm{\diver^s_\d}
        \norm{\bell^s_\d\bbf-\bg_k}_{H^{s,\d}(\O;\R^n)}\\
    &<\e_0.
\end{align*}

Finally, for \eqref{eqt: 3}, we recall that $\{\bg_k\}_{k=1}^\infty\subseteq C^\infty_c(\R^n;\R^n)$ and use the assumption that $u$ satisfies \eqref{eqt: greens}:
\[
    \int_\O\bg \cdot\grad^s_\d u\,d\bx + \int_\O u\diver^s_\d\bg_k\,d\bx 
    =\int_{\Gamma_\d} \bg_k\cdot\vcN u\,d\bx.
\]
Substituting this into~\eqref{eqt: 3} gives
\begin{align*}
    &\abs{\int_\O\bg_k \cdot\grad^s_\d u\,d\bx
        +\int_\O u\diver^s_\d\bg_k\,d\bx
        +\int_{\Gamma_\d\setminus\Gamma_{\e_1}}\bbf\cdot \vcN u\,d\bx}\\
    &\qqquad\qquad= 
    \abs{\int_{\Gamma_\d} \bg_k\cdot\vcN u\,d\bx
        -\int_{\Gamma_\d\setminus\Gamma_{\e_1}} \bbf\cdot \vcN u\,d\bx}\\
    &\qqquad\qquad=
    \abs{\int_{\Gamma_{\e_1}}\bg_k\cdot\vcN u\,d\bx
        -\int_{\Gamma_\d\setminus\Gamma_{\e_1}}(\bg_k-\bbf)\cdot\vcN u\,d\bx}\\
    &\qqquad\qquad\le 
    \norm{\bg_k}_{L^{q'}(\Gamma_{\e_1};\R^n)}
        \norm{\vcN u}_{L^q(\Gamma_{\e_1};\R^n)}\\
    &\qqqquad\qqqquad
    +\norm{\bg_k-\bell^s_\d \bbf}_{L^2(\Gamma_\d\setminus\Gamma_{\e_1};\R^n)}
        \norm{\vcN u}_{L^2(\Gamma_\d\setminus\Gamma_{\e_1};\R^n)}\\
    &\qqquad\qquad\le 2\e_0
\end{align*}
In the last step we used \eqref{eqt: p bound on normal} and the uniform boundedness of $\set{\bg_k}_{k=1}^\infty$ in $L^{q'}(\Gamma_\d;\R^n)$.

Having dealt with each term \eqref{eqt: 1}-\eqref{eqt: 3}, we conclude that $E<\e_0$. Since $\e_0$ was arbitrary, $\langle \NN u,\bbf\rangle_{L^q(\Gamma_\d;\R^n),L^{q'}(\Gamma_\d;\R^n)} = \langle \vcN u,\bbf\rangle_{L^q(\Gamma_\d;\R^n),L^{q'}(\Gamma_\d;\R^n)}$ for the given $\bbf$.

The functionals $\langle \NN u, \cdot \rangle_{L^q(\Gamma_\d;\R^n),L^{q'}(\Gamma_\d;\R^n)}$ and $\langle \vcN u,\cdot\rangle_{L^q(\Gamma_\d;\R^n),L^{q'}(\Gamma_\d;\R^n)}$ are both continuous and agree on the dense subspace $C^\infty_c(\mathop{\mathrm{Int}}\Gamma_\d;\R^n)\subseteq L^{q'}(\Gamma_\d;\R^n)$. Thus, they must agree on all of $L^{q'}(\Gamma_\d;\R^n)$, meaning $\NN u = \vcN u$ as elements of $L^q(\Gamma_\d;\R^n)$. They are therefore also equal pointwise almost everywhere. Since $\NN u\in L^2(\Gamma_\d;\R^n)$, we conclude that $\vcN u$ is as well, and the claim is proved.
\end{proof}

\begin{remark}\label{rem: normal}
    Proposition~\ref{prop: characterize N} provides some additional justification for thinking of $\NN$ as another way of making sense of the nonlocal normal, since $\NN: Z^{s,\d}(\O)\to L^2(\Gamma_\d;\R^n)$ is a bounded linear operator which agrees with $\vcN$ for functions in $C^\infty_c(\O)$. In fact, this means that $\NN|_{\mathring H^{s,\d}(\O)}: \mathring H^{s,\d}(\O)\to L^2(\Gamma_\d;\R^n)$ is the unique continuous extension of the densely defined (possibly unbounded) operator $\vcN: D(\vcN)\to L^2(\Gamma_\d;\R^n)$ with $D(\vcN) = \set{\vphi\in C^\infty_c(\O_\d): \supp\vphi\subseteq\O}$.
\end{remark}

We record an immediate consequence of Proposition~\ref{prop: characterize N} for convenient future reference.

\begin{corollary}
    If \ref{H1} holds, $u\in Z^{s,\d}(\O)$, $\norm{\vcN u}_{L^q(\Gamma_\d;\R^n)}<\infty$ for some $q\in\left(\frac{2}{1+s},2\right]$, and $u$ satisfies \eqref{eqt: greens}, then $u$ also satisfies the following distributional form of the nonlocal Green's identity:
\begin{align}\label{eqt: greensDist}
    \forall\bvphi\in C^\infty_c(\R^n;\R^n),\quad
    \int_\O\left(\grad^s_\d u\cdot \bvphi+u\diver^s_\d\bvphi\right)d\bx
    =\int_{\Gamma_\d} \bvphi\cdot\NN u\, d\bx.
\end{align}
\end{corollary}

In addition to Remark~\ref{rem: normal}, the importance of the operator $\NN$ is highlighted by the following fact: the norm of the extension-by-zero operator $E_0$ is bounded by $\norm{\NN}$, so long as \eqref{eqt: greensDist} holds.
\begin{proposition}\label{prop: grad of ext}
    Suppose \ref{H1} is satisfied. If $u\in Z^{s,\d}(\O)$ satisfies \eqref{eqt: greensDist}, then the truncated-fractional gradient $\grad^s_\d$ of the extension $E_0u$ is given by
\begin{align}\label{eqt: grad of ext}
    \grad^s_\d(E_0u)
    =\begin{cases}
        \grad^s_\d u, & \text{ in }\O,\\
        -\NN u, & \text{ in }\Gamma_\d,\\
        0, & \text{ in } (\O_\d)^c.
    \end{cases}
\end{align}
    As a result, $\norm{E_0u}_{H^{s,\d}(\R^n)}\le \sqrt{1+\norm{\NN}^2}\norm{u}_{H^{s,\d}(\O)}.$
\end{proposition}

\begin{proof}
    For a given $u\in Z^{s,\d}(\O)$, let $\bw$ denote the right-hand side of \eqref{eqt: grad of ext}.
    It suffices to show that, for any $\bvphi\in C^\infty_c(\R^n;\R^n)$,
    \begin{align*}
        \int_{\R^n}\bw\cdot \bvphi\, d\bx = -\int_{\R^n}(E_0u)\diver^s_\d\bvphi\, d\bx.
    \end{align*}
    For any $\bvphi\in C^\infty_c(\R^n;\R^n)$, we may apply \eqref{eqt: greensDist} to integrate by parts on $\O$, in terms of $\NN$:
\[
    \int_{\R^n}\bw \cdot \bvphi\, d\bx
    =\int_{\O}\bw \cdot \bvphi\, d\bx
    =\int_\O \grad^s_\d u\cdot \bvphi\,d\bx
        -\int_{\Gamma_\d} \NN u\cdot\bvphi\,d\bx
    =-\int_\O u\diver^s_\d\bvphi\, d\bx.
\]
Since $u\in Z^{s,\d}(\O)$ implies $\supp(E_0u)=\supp u\subseteq\ov{\O}$,
\begin{align*}
    \int_{\R^n}\bw \cdot \bvphi\, d\bx = -\int_\O u\diver^s_\d\bvphi\, d\bx = -\int_{\R^n}(E_0u)\diver^s_\d\bvphi\,d\bx.
\end{align*}
As this holds for all test functions $\bvphi$, we conclude that $\bw = \grad^s_\d(E_0u)$ in $\R^n$. The estimate on the norm of the extension then follows from the definition of the $H^{s,\d}(\R^n)$ norm and the above characterization of $\grad^s_\d(E_0u)$:
    \begin{align*}
        \norm{E_0u}_{H^{s,\d}(\R^n)}^2 &= \norm{E_0u}_{L^2(\R^n)}^2 + \int_{\R^n}\abs{\grad^s_\d (E_0u)}^2\,d\bx\\
        &=\norm{u}_{L^2(\O_\d)}^2 + \int_{\O}\abs{\grad^s_\d u}^2\,d\bx+ \int_{\Gamma_\d}\abs{\NN u}^2\,d\bx\\
        &\le \norm{u}_{H^{s,\d}(\O)}^2+\norm{\NN}^2\norm{u}_{H^{s,\d}(\O)}^2
    \end{align*}
    as desired.
\end{proof}

\subsection{Some Sufficient Conditions}
The distributional nonlocal normal $\NN$ appears both in the formula for the norm of the zero extension of $u\in Z^{s,\d}(\O)$ as well as the integration by parts formula \eqref{eqt: greensDist}. However, to be able to apply these results we need to satisfy the hypotheses of Proposition~\ref{prop: characterize N} and its corollary. In this subsection, we introduce two sufficient conditions that will be useful in practice. See Corollary~\ref{coro: extend} for a summary of the main conclusions of this subsection. 

We begin by introducing a new closed subspace of $Z^{s,\d}(\O)$. Roughly speaking, this set consists of limits of functions satisfying the hypotheses of Theorem~\ref{thm: by parts} and Lemma~\ref{lem: norm of normal}. The second sufficient condition is perhaps more familiar, as it is related to the classical Hardy inequality.

\begin{definition}\label{def: cA}
    Assume \ref{H1} and let $\e\in(0,\d)$ and $q\in(1,\frac1s)$ be given. We define the set $\cA^{s,\d}_{q,\e} \subseteq Z^{s,\d}(\O)$ as follows. A function $u\in Z^{s,\d}(\O)$ is in $\cA^{s,\d}_{q,\e}$ if and only if there is a sequence $\set{u_k}_{k=1}^\infty \subseteq Z^{s,\d}(\O)$ with
    \begin{align*}
        \lim_{k\to\infty}\norm{u_k-u}_{H^{s,\d}(\O)}=0 \ \ \text{ and } \ \ \forall k\in\N, \ (u_k)|_{\Gamma_{-\e}} \in L^{\qbar}(\Gamma_{-\e}),
    \end{align*}
    where $\qbar$ is given by Definition~\ref{def: qbar}.

\end{definition}

\begin{proposition}\label{prop: IBP in Asd}
Suppose \ref{H1} holds and let $\e\in(0,\d)$ and $q\in\left(1,\frac{1}{s}\right)$.  
Then $\cA^{s,\d}_{q,\e}$ is a closed subspace of $Z^{s,\d}(\O)$.  Moreover, if $q\in\left(\frac{2}{1+s},\min\set{\frac1s,2}\right)$, then any $u\in \cA^{s,\d}_{q,\e}$ satisfies the nonlocal Green's identity \eqref{eqt: greensDist}.

\end{proposition}
\begin{proof}
Showing that $\cA^{s,\d}_{q,\e}$ is a closed subspace of $Z^{s,\d}(\O)$ is straight-forward from the definitions, and we leave the details to the reader.  Next, given $u\in\cA^{s,\d}_{q,\e}$ as in the statement of the proposition, by Definition~\ref{def: cA} we may find an approximating sequence $\set{u_k}_{k=1}^\infty$ such that each $u_k$ satisfies the integrability hypotheses of Theorem~\ref{thm: by parts}. As each $u_k$ is zero on $\Gamma_\d$, $u_k$ satisfies \eqref{eqt: greens} for every $k$. Additionally, $u_k\in Z^{s,\d}(\O)$ and Lemma~\ref{lem: norm of normal} ensures that $\norm{\vcN u}_{L^{q}(\Gamma_\d;\R^n)}<\infty$. 

Thus, the hypotheses of Proposition~\ref{prop: characterize N} are satisfied and $\vcN u_k = \NN u_k$ for each $k\in\mathbb{N}$. As $\NN$ is a bounded linear operator, for any $\bvphi\in C^\infty_c(\R^n;\R^n)$, we may write
\begin{align*}
    &\abs{\int_\O\left[\grad^s_\d u\cdot \bvphi+u\diver^s_\d\bvphi\right]d\bx
        -\int_{\Gamma_\d} \bvphi\cdot\NN u\, d\bx}\\
    &\qqquad\le 
    \abs{\int_\O\left[\grad^s_\d (u-u_k)\cdot\bvphi
            +(u-u_k)\diver^s_\d\bvphi\right]d\bx
        -\int_{\Gamma_\d}\bvphi\cdot\NN(u-u_k)\,d\bx}\\
    &\qqqquad\qqquad+ 
    \abs{\int_\O\left(\grad^s_\d u_k\cdot\bvphi+u_k\diver^s_\d\bvphi\right)d\bx
        -\int_{\Gamma_\d} \bvphi\cdot\NN u_k\, d\bx}\\
    &\qqquad\le 
    \norm{u-u_k}_{H^{s,\d}(\O)}\norm{\bvphi}_{L^2(\O;\R^n)}
        +\norm{u-u_k}_{L^2(\O)}\norm{\diver^s_\d\bvphi}_{L^2(\O)}\\
    &\qqqquad\qqquad+
    \norm{\bvphi}_{L^2(\Gamma_\d;\R^n)}\norm{\NN}\norm{u-u_k}_{H^{s,\d}(\O)}.
\end{align*}
Since $u_k\to u$ in $H^{s,\d}(\O)$, the claim follows.
\end{proof}

Our next tools involve the quantity $u/d_\O^s$, which appears in the fractional Hardy inequality on bounded domains and will play an important role in our characterization of $\mathring H^{s,\d}(\O)$ in Theorem~\ref{thm: char Hsd0}.

\begin{proposition}\label{prop: bound normal in terms of hardy}
    Assume \ref{H1}. For any $p\in(1,2)$, there is some constant $C>0$ depending only on $p,n,\d,\O$ and $w_\d$ such that for all $u\in L^1(\O_\d)$ with $u/d_\O^s\in L^2(\O_\d)$, $\norm{\vcN u}_{L^p(\Gamma_\d;\R^N)}\le C\norm{u/d_\O^s}_{L^2(\O_\d)}$.
\end{proposition}
\begin{proof}
    Let $u$ and $p$ be given as in the statement of the proposition. We can rewrite the definition of $\vcN u$, \eqref{eq: nonlocal normal vector}, using the fact that $u$ is zero on $\Gamma_\d$ and $\supp w_\d\subseteq B_\d(\mathbf 0)$. Combining this with Minkowski's integral inequality gives
    \begin{align*}
        \norm{\vcN u}_{L^p(\Gamma_\d;\R^n)} &= \left(\int_{\Gamma_\d}\abs{\int_{B_\d(\mathbf0)} \frac{u(\bx-\bh)}{\abs{\bh}^{n+s}} \frac{\bh}{\abs{\bh}} w_\d(\bh)\chi_{\O}(\bx-\bh)\,d\bh}^p\,d\bx\right)^{1/p}\\
        &\le a_0\int_{B_\d(\mathbf0)}\left(\int_{\Gamma_\d} \frac{\abs{u(\bx-\bh)}^p}{\abs{\bh}^{np+sp}}\chi_{\O}(\bx-\bh)\,d\bx\right)^{1/p}\,d\bh\\
        &\le a_0\int_{B_\d(\mathbf0)}\abs{\bh}^{-n+\e}\left(\int_{\Gamma_\d} \frac{\abs{u(\bx-\bh)}^p}{\abs{\bh}^{sp+\e p}}\chi_{\O}(\bx-\bh)\,d\bx\right)^{1/p}\,d\bh,
    \end{align*}
    for any $\e>0$.
    
    The only points which contribute to the integral occur when $\bx\in\Gamma_\d$ and $\bx-\bh\in\O$, which means that for these points, $\abs{\bh}\ge d_\O(\bx-\bh)$. Hence,
    \begin{align*}
        \norm{\vcN u}_{L^p(\Gamma_\d;\R^n)} &\le a_0\int_{B_\d(\mathbf0)}\abs{\bh}^{-n+\e}\left(\int_{\Gamma_\d} \frac{\abs{u(\bx-\bh)}^p}{d_\O(\bx-\bh)^{sp+\e p}}\chi_{\O}(\bx-\bh)\,d\bx\right)^{1/p}\,d\bh\\
        &\le a_0\int_{B_\d(\mathbf0)}\abs{\bh}^{-n+\e}\left(\int_{\Gamma_{-\d}} \frac{\abs{u(\by)}^p}{d_\O(\by)^{sp+\e p}}\,d\by\right)^{1/p}\,d\bh\\
        &= \frac{a_0 \d^\e\s_{n-1}}{\e}\norm{u/d_\O^{s+\e}}_{L^p(\Gamma_{-\d})},
    \end{align*}
    For any $\e>0$. Therefore, it suffices to show that $u/d_\O^{s+\e}\in L^p(\Gamma_{-\d})$ for some $\e>0$. Since $p<2$, we can set $\e\coloneqq \frac{2-p}{4p}$. Let $q=2/p$ and let $q'$ denote its H\"older conjugate. Then $\e pq' = 1/2$ and
    \begin{align*}
        \norm{\frac{u}{d_\O^{s+\e}}}_{L^p(\Gamma_{-\d})}^p &= \int_{\Gamma_{-\d}}\frac{\abs{u}^p}{d_\O^{sp}}\frac{1}{d_\O^{\e p}}\, d\by\\
        &\le\left(\int_{\Gamma_{-\d}} \frac{\abs{u}^{pq}}{d_\O^{spq}}\, d\by\right)^{1/q}\left(\int_{\Gamma_{-\d}}d_\O^{-\e pq'}\, d\by\right)^{1/q'}\\
        &= \norm{u/d_\O^s}_{L^2(\Gamma_{-\d})}^p\left(\int_{\Gamma_{-\d}} d_\O^{-1/2}\,d\by\right)^{(2-p)/2}.
    \end{align*}
    Since $\O$ is $C^{1,1}$, Proposition~\ref{prop: distIntegrable} ensures that $d_\O^{-1/2}$ is integrable and thus the proof is complete.
\end{proof}

\begin{proposition}\label{prop: hardy implies greens}
    Assume \ref{H1}. Suppose that $u\in Z^{s,\d}(\O)$ and $u/d_\O^s\in L^2(\O_\d)$. Then \eqref{eqt: greens} holds.
    
\end{proposition}
\begin{proof}
Let $u$ be given as in the statement, and fix $p\in(1,2)$ and $\bvphi\in C^\infty_c(\R^n;\R^n)$. From Lemma~\ref{lem: approx w/cpt supp} there is a sequence $\set{u_k}_{k=1}^\infty\subseteq Z^{s,\d}(\O)$ such that $\supp(u_k)\subseteq \O_{-1/k}$ for each $k$ and
\begin{align*}
    \lim_{k\to\infty}\left(k^s\norm{u-u_k}_{L^2(\O)}\right)=0= \lim_{k\to\infty}\norm{\grad^s_\d(u-u_k)}_{L^p(\O;\R^n)}. 
\end{align*}
For convenience, let us formally write
\[
    F(v,\bvphi)
    :=\int_\O\left(\grad^s_\d v\cdot\bvphi + v\diver^s_\d\bvphi\right)d\bx
\]
Then for any $k\in\N$,
\begin{equation}\label{eqt: hardy ibp estimate}
    \abs{F(u,\bvphi)-\int_{\Gamma_\d}\bvphi\cdot\vcN u\,d\bx}
    \le 
    \abs{F(u-u_k,\bvphi)}
    +\abs{F(u_k,\bvphi)-\int_{\Gamma_\d}\bvphi\cdot\vcN u\,d\bx}.
\end{equation}
Since $u_k\to u$ in $L^2(\O)$ and $\grad^s_\d (u-u_k)\to 0$ in $L^p(\O;\R^n)$, we may use H\"older's inequality to conclude that $\lim_{k\to\infty}|F(u-u_k,\bvphi)|=0$. For the second, we note that each $u_k$ is zero on a neighborhood of $\partial\O$. Thus, with $q_1\in(1,1/s)$ and $\e=1/k$ fixed, we may apply Theorem~\ref{thm: by parts}, as $(u_k)|_{\Gamma_{\pm1/k}}\in L^\infty(\Gamma_{\pm1/k})\subseteq L^{\overline{q_1}}(\Gamma_{\pm 1/k})$. Moreover, $\bvphi$ has the required integrability so that $u_k$ satisfies \eqref{eqt: greens}:
\[
    F(u_k,\bphi)
    =\int_{\Gamma_\d}\bvphi\cdot\vcN u_k\,d\bx.
\]
Hence, it suffices to show that $\norm{\vcN(u-u_k)}_{L^p(\Gamma_\d;\R^n)}\to0$ as $k\to\infty$. Applying Proposition~\ref{prop: bound normal in terms of hardy}, we have
\begin{align*}
    \norm{\vcN (u-u_k)}_{L^p(\Gamma_\d)}\le C\norm{\frac{u-u_k}{d_\O^s}}_{L^2(\O)}.
\end{align*}
To see that this tends to zero, recall that $u_k$ is zero on $\Gamma_{-1/k}$ and so
\begin{align*}
    \int_{\O} \frac{\abs{u-u_k}^2}{d_\O^{2s}}\,d\bx &= \int_{\O_{-1/k}} \frac{\abs{u-u_k}^2}{d_\O^{2s}}\,d\bx + \int_{\Gamma_{-1/k}}\frac{\abs{u}^2}{d_\O^{2s}}\,d\bx\\
    &\le k^{2s}\int_{\O_{-1/k}}\abs{u-u_k}^2\,d\bx + \int_{\Gamma_{-1/k}}\frac{\abs{u}^2}{d_\O^{2s}}\,d\bx.
\end{align*}
The first term goes to zero, due to the fact that the sequence $\set{u_k}_{k=1}^\infty$ was selected to ensure that $k^s\norm{u-u_k}_{L^2(\O)}$ converges to zero. The second term also goes to zero since $u/d_\O^s\in L^2(\O)$ and measure of the domain of integration shrinks to zero.
Thus, the left-hand side of \eqref{eqt: hardy ibp estimate} can be made arbitrarily small, implying that
\begin{align*}
    F(u,\bvphi)
    =\int_\O\left(\grad^s_\d u\cdot\bvphi+u\diver^s_\d\bvphi\right)d\bx
    =\int_{\Gamma_\d}\bvphi\cdot\vcN u\,d\bx,
\end{align*}
as desired.
\end{proof}

We summarize the last few results.
\begin{corollary}\label{coro: extend}
    If \ref{H1} holds, $u\in Z^{s,\d}(\O)$, and either
    \begin{itemize}
        \item[(i)] there is some $\e\in(0,\d)$ and $q\in\left(\frac{2}{1+s},\min\set{\frac1s,2}\right)$ such that $u\in \cA^{s,\d}_{q,\e}$, or
        \item[(ii)] $u/d_\O^s\in L^2(\O_\d)$,
    \end{itemize}
    then $u$ satisfies \eqref{eqt: greensDist} and $E_0u\in H^{s,\d}(\R^n)$ with $\norm{E_0u}_{H^{s,\d}(\R^n)}\le C\norm{u}_{H^{s,\d}(\O)}$ for a constant $C$ depending only on the norm of the lift $\bell^s_\d: L^2(\Gamma_\d;\R^n)\to H^{s,\d}(\O;\R^n)$.
\end{corollary}
\begin{proof}
    In the first case, the assumption that $u\in\cA^{s,\d}_{q,\e}$ ensures \eqref{eqt: greensDist} holds from Proposition~\ref{prop: IBP in Asd}. Proposition~\ref{prop: grad of ext} then gives $E_0u\in H^{s,\d}(\R^n)$.

    For the second case, Propositions~\ref{prop: hardy implies greens} and \ref{prop: bound normal in terms of hardy} ensure that the hypotheses of Proposition~\ref{prop: characterize N} are satisfied. The desired conclusion then immediately follows, with the help of Proposition~\ref{prop: grad of ext}.
\end{proof}

\section{\texorpdfstring{
The Closure of $C^\infty_c(\Omega)$ in $H^{s,\delta}(\Omega)$
}{
The Closure of Compactly Supported Smooth Functions
}}
Recall that Definition~\ref{def: mathringH} introduced $\mathring H^{s,\d}(\O)$. The main result of this section (Theorem~\ref{thm: char Hsd0}) shows that $\mathring H^{s,\d}(\O)$ satisfies many, but not all, of the truncated-fractional analogs of the defining characteristics of $H^1_0(\O)$ (Theorem~\ref{thm: char H10}). We present several important remarks following its proof, but first we clarify an important difference between the current truncated-fractional and more traditional fractional Sobolev spaces.

\subsection{Relationship to Other Fractional Sobolev Spaces}
To clarify the relationship between $\mathring H^{s,\d}(\O)$, $\widetilde H^{s,\d}(\O)$, and some alternative notions of fractional Sobolev space, we need some additional notation. The spaces described below have been well-studied and will not play a terribly important role here. Thus, we do not provide much introduction or a complete definition and instead refer to \cite{triebel_interpolation_1978,mclean_strongly_2000,grisvard_elliptic_1985} for additional details.

\begin{definition}\label{def: compVal}
    For $\O\subseteq\R^n$ an open, bounded, connected, Lipschitz domain, we introduce the following complementary-value spaces:
    \begin{enumerate}
        \item $\widetilde H^{s}(\O) \coloneqq \set{u\in H^{s}(\R^n): \supp u\subseteq \overline\O}$, where $H^s(\R^n)$ is the Bessel potential space defined in terms of the Fourier transform. This is a Banach space with the norm $\norm{u}_{\widetilde H^{s}(\O)}=\norm{ u}_{H^{s}(\R^n)}$.
        \item $\widetilde B^s_{2,2}(\O) \coloneqq \set{u\in B^s_{2,2}(\R^n): \supp u\subseteq \overline\O}$, where $B^s_{2,2}(\R^n)$ is the Besov space, which may defined explicitly or as an interpolation space. This is a Banach space with the norm $\norm{u}_{\widetilde B^s_{2,2}(\O)}=\norm{ u}_{B^s_{2,2}(\R^n)}$.
        \item $\widetilde W^{s,2}(\O) \coloneqq \set{u\in W^{s,2}(\R^n): \supp u\subseteq \overline\O}$, where $W^{s,2}(\R^n)$ is the Sobolev-Slobodeckij space defined in terms of the Gagliardo seminorm. This is a Banach space with the norm $\norm{u}_{\widetilde W^{s,2}(\O)}=\norm{ u}_{W^{s,2}(\R^n)}$.
    \end{enumerate}
\end{definition}

\begin{remark}\label{rem: fracSoboSp}
    In general, there are many competing notions of what a ``Fractional Sobolev Space" should be. For example, some approaches are motivated by interpolation, which requires a choice between the real and complex methods. We assume $\O$ is at least Lipschitz and, in this paper, we consider only the case where $p=2$. Thus, the real and complex methods of interpolation coincide (see, for example, \cite{chandler-wilde_interpolation_2015}). Furthermore, in this setting the three spaces just introduced are identical as sets with equivalent norms. Thus, in our case the various competing fractional Sobolev spaces are all the same.
    
    Each of the complementary-value spaces above is also equal to the appropriate norm-closure of $C^\infty_c(\O)$. See, for example, Theorem 3.29 of \cite{mclean_strongly_2000}, Theorem 4.3.2/1 of \cite{triebel_interpolation_1978} and Theorem 1.4.2.2 of \cite{grisvard_elliptic_1985}. Density of test functions in complementary-value spaces has been explored elsewhere; see for example Theorem 3.9 of \cite{bellido_nonlocal_2025} or the paper \cite{fiscella_density_2015}.

    Similar to Definition~\ref{def: widetildeH}, Definition~\ref{def: compVal} uses the norms inherited by viewing the function as defined on the full space $\R^n$, rather than a norm intrinsic to $\O$ as in Definition~\ref{def: mathringH}. The latter choice would instead lead to the spaces often denoted $\mathring H^s(\O), \mathring B^s_{2,2}(\O)$, or $\mathring W^{s,2}(\O)$ \cite{triebel_interpolation_1978,mclean_strongly_2000}.
\end{remark}

\begin{remark}\label{rem: equivNorms}
    Extending by zero on $\Gamma_\d$, we may consider $C^\infty_c(\O)$ as a subspace of $L^2(\O_\d)$. We consider two norms on this space:
    \begin{align*}
        \norm{\varphi}_{H^{s,\d}(\O)}&\coloneqq \left(\norm{\varphi}_{L^2(\O_\d)}^2+ \norm{\grad^s_\d \varphi}_{L^2(\O;\R^n)}^2\right)^{1/2},\\
        \norm{\varphi}_{\widetilde H^{s,\d}(\O)}&\coloneqq \left(\norm{E_0\vphi}_{L^2(\R^n)^2} + \norm{\grad^s_\d(E_0\vphi)}_{L^2(\R^n;\R^n)}\right)^{1/2}. 
    \end{align*}
    These correspond to viewing $C^\infty_c(\O)$ as a subspace of $H^{s,\d}(\O)$ or $H^{s,\d}(\R^n)$, respectively (compare Definitions~\ref{def: mathringH} and \ref{def: widetildeH}). 
    
    By analogy with with the spaces described in Definition~\ref{def: compVal}, one might not expect $\norm{\cdot}_{H^{s,\d}(\O)}$ and $\norm{\cdot}_{\widetilde H^{s,\d}(\O)}$ to be equivalent in general, even on extension domains. Indeed, for $s=1/2$ the analogous versions of $\norm{\cdot}_{H^{s,\d}(\O)}$ and $\norm{\cdot}_{\widetilde H^{s,\d}(\O)}$ are not equivalent norms on $C^\infty_c(\O)$ in any of the three cases (Bessel potential, Besov, or Sobolev-Slobodeckij spaces). In symbols:
    \begin{align*}
        s=\frac12\implies \widetilde H^{s}(\O)\subsetneqq \mathring H^s(\O), \widetilde B^{s}_{2,2}(\O)\subsetneqq \mathring B^{s}_{2,2}(\O), \text{ and } \widetilde W^{s,2}(\O)\subsetneqq \mathring W^{s,2}(\O);
    \end{align*}see, for example, Theorem 3.33 of \cite{mclean_strongly_2000}, Remark 2 of \cite[Section 4.3.2]{triebel_interpolation_1978}, or Section 1.4 of \cite{grisvard_elliptic_1985}.

    Perhaps surprisingly, the truncated-fractional space $\mathring H^{s,\d}(\O)$ behaves differently. Proposition~\ref{prop: RingToTildeHomeo} shows that even for $s=1/2$, $\widetilde H^{s,\d}(\O)$ and $\mathring H^{s,\d}(\O)$ are essentially the same space whenever $\O$ has $C^{1,1}$ boundary.  One may understand this intuitively as follows: for the Sobolev-Slobodeckij spaces, the analogue of the $H^{s}(\O)$-norm only ``sees'' $\O$, while the $\widetilde H^{s}(\O)$ norm ``sees'' the function as it crosses the boundary. Thus the norm $\norm{\cdot}_{\widetilde H^s(\O)}$ is sensitive to potential discontinuities or sharp corners introduced at the boundary when extending functions from $\O$ to all of $\R^n$. See, for example, \cite{dyda_function_2019}. Contrast this with the behavior of the present framework. The nonlocality of the operator $\grad^s_\d$ allows it to ``see'' outside of $\O$, for both the $H^{s,\d}(\O)$- and $\widetilde H^{s,\d}(\O)$-norm. Both norms already account for any potentially sharp transition caused by the extension, and hence the above non-equivalence at $s=1/2$ does not hold for these spaces.
\end{remark}

\begin{proposition}\label{prop: RingToTildeHomeo}
    If \ref{H1} is satisfied, then the zero-extension $E_0$ acts as a linear homeomorphism from $\mathring H^{s,\d}(\O)$ to $\widetilde H^{s,\d}(\O)$.
\end{proposition}
\begin{proof}
    For any $u\in\mathring H^{s,\d}(\O)$, there is a sequence $\set{u_k}_{k=1}^\infty\subseteq C^\infty_c(\O_\d)$ such that $\supp(u_k)\subseteq \O$ for each $k$ and $u_k\to u$ in $H^{s,\d}(\O)$. For any choice of $q$ in the interval $\left(\frac{2}{1+s},\min\set{\frac1s,2}\right)$ and $\e\in(0,\d)$, each $u_k|_{\Gamma_{-\e}}$ is in $L^\infty(\Gamma_{-\e})\subseteq L^\qbar(\Gamma_{-\e})$. So $u\in \cA^{s,\d}_{q,\e}$ and Corollary~\ref{coro: extend} ensures $E_0u$ is in $H^{s,\d}(\R^n)$ with $\norm{E_0 u}_{H^{s,\d}(\R^n)}\le C\norm{u}_{H^{s,\d}(\O)}$. Each $u_k$ is also in $\cA^{s,\d}_{q,\e}$, meaning that $u-u_k$ is as well. Thus,
    \begin{align*}
        \norm{E_0u-  E_0u_k}_{H^{s,\d}(\R^n)}\le C\norm{u-u_k}_{H^{s,\d}(\O)}\to 0.
    \end{align*}
    Since $\set{E_0u_k}_{k=1}^\infty$ is a sequence in $C^\infty_c(\R^n)$ such that each $E_0u_k$ has support in $\O$, we have shown that $u\in \widetilde H^{s,\d}(\O)$ by directly verifying the definition. This shows that $E_0$ sends $\mathring H^{s,\d}(\O)$ into $\widetilde H^{s,\d}(\O)$ and we have already shown that $E_0: \mathring H^{s,\d}(\O)\to H^{s,\d}(\R^n)$ is a bounded linear operator in Corollary~\ref{coro: extend}.

    Similar reasoning shows that the restriction map $R: \widetilde H^{s,\d}(\O)\to  L^2(\O_\d)$ given by $Ru \coloneqq u|_{\O_\d}$ maps into $\mathring H^{s,\d}(\O)$. Additionally, $R$ acts as a continuous inverse to $E_0$ with $\norm{Ru}_{H^{s,\d}(\O)}\le \norm{u}_{\widetilde H^{s,\d}(\O)}$.
\end{proof}

\subsection{\texorpdfstring{
Characterizing $\mathring H^{s,\d}(\O)$
}{
Characterizing the Zero-Boundary Space
}}
We now present a truncated-fractional analog of Theorem~\ref{thm: char H10}. We note that the equalities in \eqref{eqt: equiv}, but not the first isomorphism, are already known; see \cite{cueto_variational_2023} and \cite{triebel_interpolation_1978}.

\begin{theorem}\label{thm: char Hsd0}
    If \ref{H1} holds, then the following are equivalent characterizations of the space $\mathring H^{s,\d}(\O)$.
    \begin{enumerate}[label=(TF\arabic*)]
    \item $\mathring H^{s,\d}(\O)$ is the closure of $C^\infty_c(\O)$ with respect to the $H^{s,\d}(\O)$-norm. \label{property: (s,d) closure of test functions}

    \item $\mathring H^{s,\d}(\O) = \set{u|_{\O_\d} : u\in H^{s,\d}(\R^n), \ \supp u\subseteq \overline\O}$. Furthermore, $\mathring H^{s,\d}(\O)$ can be identified with the following complementary-value spaces (which are identical as sets and with equivalent norms) via the homeomorphism $E_0$:
        \begin{align}
            \mathring H^{s,\d}(\O)\cong E_0(\mathring H^{s,\d}(\O)) = \widetilde H^{s,\d}(\O) = \widetilde H^s(\O) = \widetilde B^s_{2,2}(\O)=\widetilde W^{s,2}(\O). \label{eqt: equiv}
        \end{align}\label{property: (s,d) zero extension}

    \item $\mathring H^{s,\d}(\O)$ is the space of functions $u\in Z^{s,\d}(\O)$ such that $u/d_\O^s \in L^2(\O_\d)$.\label{property: (s,d) hardy}

    \item $\mathring H^{s,\d}(\O)$ is the closed subspace of $Z^{s,\d}(\O)$ where the distributional nonlocal Green's identity \eqref{eqt: greensDist} holds.
    \label{property: (s,d) IBP}
    \end{enumerate}
    Moreover, $\mathring H^{s,\d}(\O)$ has these additional properties:
    \begin{enumerate}[label=(TF5\roman{enumi})]
        \item If $u\in \mathring H^{s,\d}(\O)$ and $\grad^s_\d u = 0$ in $\O$, then $u=0$ in $\O_\d$. \label{property: (s,d) injective}
        \item There exists a constant $C<\infty$ such that, for all $u\in \mathring H^{s,\d}(\O)$, we have $\norm{u}_{L^2(\O_\d)}\le C\norm{\grad^s_\d u}_{L^2(\O;\R^n)}$. \label{property: (s,d) poincare}
        
    \end{enumerate}
\end{theorem}

\begin{proof}
    Statement~\ref{property: (s,d) closure of test functions} is the definition of $\mathring H^{s,\d}(\O)$.

    For \ref{property: (s,d) zero extension}, recall that Proposition~\ref{prop: RingToTildeHomeo} already established that the operator $E_0: \mathring H^{s,\d}(\O)\to \widetilde H^{s,\d}(\O)$ is a homeomorphism. Roughly speaking, $\widetilde H^{s,\d}(\O)$ is the closure of $C^\infty_c(\O)$ in the $H^{s,\d}(\R^n)$ norm. But Lemma 5 of \cite{cueto_variational_2023} showed that this norm is equivalent to the $H^s(\R^n)$ norm, meaning that $\widetilde H^{s,\d}(\O)$ is also the closure of $C^\infty_c(\O)$ with respect to the $H^s(\R^n)$ norm. As noted in Remark~\ref{rem: fracSoboSp}, the assumption on the boundary regularity of $\O$ ensures the closure of $C^\infty_c(\O)$ in $H^s(\R^n)$ is $\widetilde H^s(\O)$. Hence,
    \begin{align*}
        \widetilde H^{s,\d}(\O) = \overline{C^\infty_c(\O)}^{H^{s,\d}(\R^n)}= \overline{C^\infty_c(\O)}^{H^{s}(\R^n)} = \widetilde H^s(\O).
    \end{align*}
    The other equivalences claimed in \eqref{eqt: equiv} then follow from the facts recalled in Remark~\ref{rem: fracSoboSp} and references therein.

    Next, we turn to \ref{property: (s,d) hardy}. If $u\in\mathring H^{s,\d}(\O)$, then \ref{property: (s,d) zero extension} implies that $E_0u\in \widetilde W^{s,2}(\O)$. Recall also that $\O$ has $C^{1,1}$ boundary. We may therefore apply classical results such as Remark 4.3.2/2 (especially Equation (7)) and Lemma 3.2.6/1 of \cite{triebel_interpolation_1978} or Lemma 1.3.2.6 of \cite{grisvard_elliptic_1985} to see
    \begin{align*}
        \norm{u/d_\O^s}_{L^2(\O_\d)}=\norm{u/d_\O^s}_{L^2(\O)}<\infty.
    \end{align*}
    Conversely, if $u\in Z^{s,\d}(\O)$ with $u/d_\O^s\in L^2(\O_\d)$, then Corollary~\ref{coro: extend} implies $E_0 u\in H^{s,d}(\R^n)$. By \ref{property: (s,d) zero extension}, we conclude $u\in\mathring H^{s,\d}(\O)$.

    For \ref{property: (s,d) IBP}, Proposition~\ref{prop: characterize N} says that any function in $Z^{s,\d}(\O)$ satisfying \eqref{eqt: greensDist} may be extended by zero to an element of $H^{s,\d}(\R^n)$. By \ref{property: (s,d) zero extension}, any such function is in $\mathring H^{s,\d}(\O)$. Conversely, any $u\in \mathring H^{s,\d}(\O)$ is in $\cA^{s,\d}_{q,\e}$ and $u/d_\O^s\in L^2(\O_\d)$. Hence, Corollary~\ref{coro: extend} asserts that $u$ satisfies \eqref{eqt: greensDist}.

    For the additional properties, first suppose that $u\in\mathring H^{s,\d}(\O)$ and $\grad^s_\d u=0$ in $\O$. Then, by definition of $\NN u$, we have $\NN u=0$ in $\Gamma_\d$. Since $u\in\mathring H^{s,\d}(\O)$, Proposition~\ref{prop: grad of ext} gives a formula for the gradient of $E_0$, which must be zero a.e. Thus, $E_0u\in H^{s,\d}(\R^n)$ with $\grad^s_\d(E_0u)=0$ a.e., meaning that $u$ is the zero element of $\mathring H^{s,\d}(\O)$, which follows from the Fourier characterization of $\grad^s_\d: H^{s,\d}(\R^n)\to L^2(\R^n)$ as a symbol vanishing only at the origin.

    For Property~\ref{property: (s,d) poincare}, Remark 4.6 of \cite{kreisbeck_non-constant_2024}, and the fact that $\mathring H^{s,\d}(\O)$ is closed in $H^{s,\d}(\O)$ and compactly embedded in $L^2(\O)$ all together imply that $\mathring H^{s,\d}(\O)$ satisfies a fractional Poincar\'e-type inequality. This, in turn, implies~\ref{property: (s,d) injective} completing the proof.
\end{proof}

\begin{remark}
    The theorem above is quite similar to the classical result, Theorem~\ref{thm: char H10}, but with some important differences. For one, the implications of \ref{property: (s,d) hardy} are quite dependent on $s$, as discussed in the next remark. Additionally, the nonlocal integration by parts rule, \eqref{eqt: greensDist}, does not have zero on the right-hand side. This is an unavoidable consequence of allowing for nonlocal interactions across the boundary. If we restrict to testing against $\bv$ which are sufficiently regular but are also zero on $\Gamma_\d$, then the nonlocal interaction term on the right-hand side of \eqref{eqt: greensDist} vanishes.

    Perhaps the most striking difference is that Theorem~\ref{thm: char Hsd0} contains no fractional analog of \ref{property: trace zero}. In the classical setting, enforcing a zero-trace condition is the most direct way of translating physical boundary conditions into the mathematical framework. Theorem~\ref{thm: char H10} shows that just assuming the trace is zero is enough to automatically gain many other nice properties. This is not true in the truncated-fractional case, as will be further discussed in Section~\ref{sec: zeroCollar}.
\end{remark}

\begin{remark}
    Under the hypotheses of Theoerem~\ref{thm: char Hsd0}, we have
    \[
    s<\frac12 \implies \mathring H^{s,\d}(\O)\cong E_0(\mathring H^{s,\d}(\O))=\widetilde W^{s,2}(\O) =E_0( \mathring W^{s,2}(\O) )= E_0(W^{s,2}(\O)).
    \]
    This is because \ref{property: (s,d) zero extension} says, roughly, that the closure of $C^\infty_c(\O)$ is the same under any of the following norms: $\norm{\cdot}_{H^{s,\d}(\O)}$, $\norm{\cdot}_{H^{s,\d}(\R^n)}$, and $\norm{\cdot}_{W^{s,2}(\R^n)}$. When $s<1/2$, we further know that $\widetilde W^{s,2}(\O) = \mathring W^{s,2}(\O) = W^{s,2}(\O)$ (see \cite[Theorem 3.33]{mclean_strongly_2000} or \cite[Theorem 4.3.2/1]{triebel_interpolation_1978}). Additional relevant discussion of fractional Sobolev spaces, trace theorems, and extension operators may be found in \cite{dyda_function_2019}.

    One consequence of this identification is highlighted in Lemma 13 of \cite{dyda_density_2022}: when $s<1/2$, the indicator function $\chi_\O\in L^2(\O_\d)$ is in $\mathring H^{s,\d}(\O)$. Thus, jump discontinuities across the boundary are possible in $\mathring H^{s,\d}(\O)$ whenever $s<1/2$. This can also be seen from \ref{property: (s,d) hardy}: when $s<1/2$, there is no need for $u(\bx)$ to go to zero as $\bx\to\partial\O$ in order for $u/d_\O^s$ to be square-integrable. This is of course consistent with the idea that the local notion of ``zero-trace" is only sensible when working in spaces with at least $s>1/2$ order of differentiability.
\end{remark}

\begin{remark}
    Compare the definition of $\mathring H^{s,\d}(\O)$ with the definition of $\mathcal A^{s,\d}_{q,\e}$. The former requires an approximating sequence in $C^\infty_c(\O)$, while the latter requires approximation only by $Z^{s,\d}(\O)$ functions which satisfy an integrability criterion on some inner tubular neighborhood of $\partial\O$. A priori, the former is a much more restrictive requirement. However, Theorem~\ref{thm: char Hsd0} shows that the two are equivalent: if $u\in\cA^{s,\d}_{q,\e}$, then Corollary~\ref{coro: extend} implies $E_0 u\in \widetilde H^{s,\d}(\O)$, meaning $u\in\mathring H^{s,\d}(\O)$. Thus, Theorem~\ref{thm: char Hsd0} allows us to replace approximation by test functions with an apparently weaker condition that may be easier to check in practice.
\end{remark}

\begin{remark}
    In Lemma 5 of \cite{cueto_variational_2023}, it is shown that
    \[
    \overline{\set{u\in C^\infty_c(\R^n): \supp u\subseteq \O}}^{H^{s,\d}(\R^n)} = \widetilde H^s(\O_{-\d})
    \]
    when $\O_{-\d}$ has sufficient regularity.
    Later, Theorem 3.9(iii) of \cite{bellido_nonlocal_2025}, established
    \[
    \widetilde H^{s,\d}(\O) =\set{u\in H^{s,\d}(\R^n): \supp u\subseteq\overline\O}.
    \]
    Our \ref{property: (s,d) zero extension} connects these results to the space $\mathring H^{s,\d}(\O)$, which is the closure of test functions in $H^{s,\d}(\O)$ rather than $H^{s,\d}(\R^n)$. Additionally, we avoid working with $\O_{-\d}$. This means there are no explicit constraints on what values functions may attain within $\O$, and allows us to consider more general domains $\O$.

    One benefit of identifying $\mathring H^{s,\d}(\O)$ with $\widetilde H^s(\O)$ is that the dual space of $\widetilde H^s(\O)$ is well understood (see \cite{chandler-wilde_interpolation_2015,mclean_strongly_2000}; note also a correction to these in \cite{chandlerwilde_corrigendum_2022}). Additionally, this identification allows for the application of results from the theory of interpolation spaces to be applied to $\mathring H^{s,\d}(\O)$.
\end{remark}

\section{\texorpdfstring{
The Kernel of the Restrict-to-the-Collar Operator:
$Z^{s,\d}(\O)$
}{
The Kernel of the Restrict-to-the-Collar Operator
}}
\label{sec: zeroCollar}
The space $H^1_0(\O)$ is often understood as the technical, precise way of making sense of Sobolev functions which are zero on the boundary. From this viewpoint, the most direct nonlocal analog would be the space of fractional-Sobolev functions which are zero on the collar. Indeed, many papers consider the collar as a ``thick boundary", a set of positive measure where volume constraints are imposed in analogy with classical boundary conditions. In this section, we explore the space $Z^{s,\d}(\O)$. Definition~\ref{def: Zsd} defines this space as the subspace of $H^{s,\d}(\O)$ consisting of functions which are zero a.e. on the outer collar $\Gamma_\d$.

The space $Z^{s,\d}(\O)$ is motivated by property~\ref{property: trace zero} of $H^1_0(\O)$: $Z^{s,\d}(\O)$ is the nonlocal version of zero-trace Sobolev functions. However, in the current setting this space does \textbf{not} coincide with the closure of test functions (property~\ref{property: closure of test functions}). Among other things, this means that a bounded extension-by-zero operator cannot exist on all of $Z^{s,\d}(\O)$. This does not contradict the result of \cite{kreisbeck_non-constant_2024} which states that extensions are always possible modulo $N^{s,\d}(\O)$; here we do not mod out by $N^{s,\d}(\O)$, and we consider only the zero extension.

We first point out an immediate consequence of Theorem~\ref{thm: psi}, which we use to build an example of $h\in Z^{s,\d}(\O)$ which cannot be approximated by $C^\infty_c(\O)$.

\begin{corollary}[\cite{kreisbeck_non-constant_2024}]
    If \ref{H1} holds, then the intersection $N^{s,\d}(\O)\cap Z^{s,\d}(\O)$ is a one dimensional vector space. In particular, it is not the trivial space $\set{0}$.
\end{corollary}
\begin{proof}
    The paper \cite{kreisbeck_non-constant_2024} completely characterizes $N^{s,\d}(\O)$ in the current setting. One version is quoted above as Theorem~\ref{thm: psi}, which gives one of several routes to prove the current claim. For example, the inverse map $\left(\Psi^s_\d\right)^{-1}: \R\times L^p(\Gamma_\d)\to N^{s,\d}(\O)$ is a vector space isomorphism and so $\left(\Psi^s_\d\right)^{-1}\left(\cdot,0\right)$ is one-dimensional.
\end{proof}

\begin{proposition}\label{prop: notDense}
    If \ref{H1} holds, then there are elements of $Z^{s,\d}(\O)$ which may not be approximated in $H^{s,\d}(\O)$ by test functions supported in $\O$. In other words, $\set{\vphi\in C^\infty_c(\O_\d): \supp\vphi\subseteq\O}$ is not dense in $Z^{s,\d}(\O)$.
\end{proposition}
\begin{proof}
    Pick some nonzero $u\in N^{s,\d}(\O)\cap Z^{s,\d}(\O)$. Assume for the sake of contradiction that there is a sequence of functions $\set{u_k}_{k=1}^\infty\subseteq C^\infty_c(\O_\d)$ with $\supp u_k\subseteq\O$ and $u_k\to u$ in $H^{s,\d}(\O)$. Then $u\in \mathring H^{s,\d}(\O)$, by definition. Theorem~\ref{thm: char Hsd0} and Proposition~\ref{prop: grad of ext} may be applied to characterize the nonlocal gradient of $E_0u$ in $\R^n$ in terms of $\NN u$ as defined Definition~\ref{def: NN}. Since $u\in N^{s,\d}(\O)$, we have $\grad^s_\d u=\mathbf 0$ in $\O$ and $\NN u=\mathbf 0$ in $\Gamma_\d$. Thus, $\grad^s_\d(E_0 u)$ is zero pointwise a.e. in $\R^n$. This is only possible if $E_0 u$ is identically zero. In particular, $u$ must have been zero in $\O_\d$ and we have obtained our contradiction.
\end{proof}

It is now clear that $Z^{s,\d}(\O)$ and $\mathring H^{s,d}(\O)$ are distinct spaces; for example, the former contains nonzero elements of $N^{s,\d}(\O)$ while the latter does not. As a consequence, $Z^{s,\d}(\O)$ lacks the properties characterizing $\mathring H^{s,\d}(\O)$ in Theorem~\ref{thm: char Hsd0}.
\begin{corollary}\label{coro: Zsd}
    Assume \ref{H1}. Then there is a nonzero function $u_*\in Z^{s,\d}(\O)$ such that
    \begin{enumerate}
        \item $u_*$ is not in $\mathring H^{s,\d}(\O)$ and therefore cannot be realized as an $H^{s,\d}(\O)$-limit of functions in $C^\infty_c(\O)$.
        \item $E_0u_*\notin H^{s,\d}(\R^n)$, meaning that $E_0$ does not map $Z^{s,\d}(\O)$ into $H^{s,\d}(\R^n)$.
        \item The ratio $u_*/d_\O^s$ is not in $L^2(\O_\d)$, implying that $u_*$ need not decay to zero or even remain bounded as $\bx$ approaches $\partial\O$.
        \item $u_*$ fails the distributional nonlocal Green's identity \eqref{eqt: greensDist}.
        \item $\grad^s_\d u_*=\mathbf 0$ in $\O$ even though $u_*\not\equiv0$ in $\O_\d$.
        \item $u_*$ violates the fractional Poincar\'e identity: for all
        $C>0$, \[\norm{u_*}_{L^2(\O_\d)}>C\norm{\grad^s_\d u_*}_{L^2(\O;\R^n)}.\]
    \end{enumerate}
\end{corollary}

\begin{example}\label{ex: ustar}
    We provide a numerical approximation to a function $u_*$ given in Corollary~\ref{coro: Zsd}. We set $n=1$, $\O=(-1,1)$, $s=1/4$, $\d=1/10$, $b_0=1/3$, and 
    \begin{minipage}{0.5\textwidth}
    \begin{align*}
        \bar w_\d(r) = \begin{cases}
            1, & 0\le r \le b_0\d,\\
            \frac{g(\d-r)}{ g(\d-r) + g(r-b_0\d) }, & b_0\d<r<\d,\\
            0, & r\ge \d,
        \end{cases}
    \end{align*}
    \end{minipage}
    \begin{minipage}{0.48\textwidth}

\includegraphics[width=0.8\textwidth]{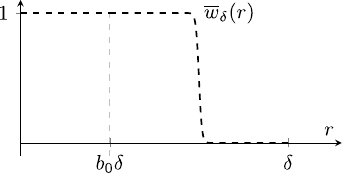}

\end{minipage}

    \noindent
    where $g(t) = e^{-1/t}$.
    Roughly speaking, we construct an approximation to $u_*\in Z^{s,\d}(\O)$ by solving a discretized version of the linear system $\grad^s_\d u_* = \mathbf 0$; some details are provided in Appendix~\ref{app: numerics}. The result is plotted in Figure~\ref{fig: ustar}

    The function $u_*$ is symmetric about the origin, and it is nearly constant on the interval $[-0.8,0.8]$:
    \begin{align*}
        \left(\max_{x\in[-0.8,0.8]} u_*(x)\right)-\left(\min_{x\in[-0.8,0.8]} u_*(x)\right)\approx3.1148\times 10^{-6}.
    \end{align*}
    Since $\d=0.1$, the region $[-0.8,0.8]$ is not sensitive to the behavior near the endpoints. However, closer to the boundary, $u_*$ behaves like a singular power function. For example, take $c = 1-e^{-5.5}\approx 0.9959$ and consider the interval $(c,1)$. Our highest resolution discretization included $N=6145$ points on $[-1,1]$, and $176$ of those points lie in $(c,1)$. Using these points, we fit a power function by taking logarithms and performing a linear least squares regression. We obtained
    \begin{align*}
        u_*(x)\approx 0.0767(1-x)^{-0.37036},  \ \ x\in(c,1).
    \end{align*}
    with $R^2=0.999921$. Different choices of $c$ and different function-fitting techniques yield slightly different values; however, all the numerical evidence suggests that the exponent lies in the interval $[-.3737,-.3677]$. Thus, $u_*\in L^2(\O_\d)$, but $u/d_\O^s\notin L^2(\O_\d)$. We also note that any exponent in this range would imply $u_*\in L^{8/3}(\O_\d)$. This is consistent with \cite{kreisbeck_non-constant_2024}, which predicts that $u_*\in L^q(\O_\d)$ for $q<\frac{p}{1-s} = \frac2{1-1/4} = \frac83$. 
    
    \begin{figure}[!ht]
    \centering
    \includegraphics[width=\linewidth]{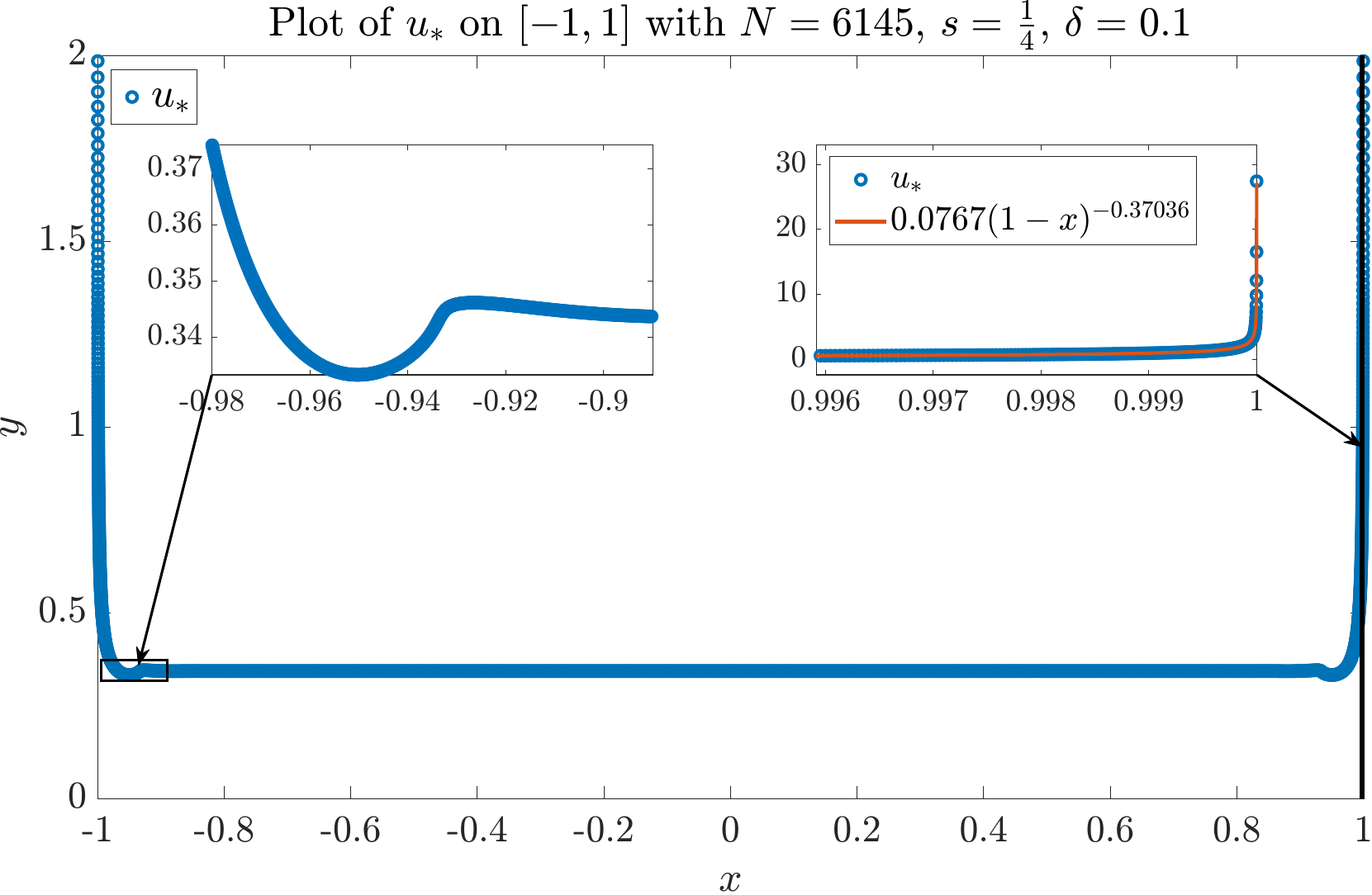}
    \caption{A plot of the numerical approximation of $u_*$ described in Example~\ref{ex: ustar} and Appendix~\ref{app: numerics}.}
    \label{fig: ustar}
\end{figure}
\end{example}

\begin{remark}
    Let us clarify an important point regarding the choice of $p=2$. The theory developed here would look quite different if some additional integrability were required.
    
    \begin{enumerate}[label=(\alph*)]
        \item The distributional nonlocal normal $\NN$, defined in terms of the lift $\bell^s_\d$, played a key role in the proof of Theorem~\ref{thm: char Hsd0}. But the lift $\bell^s_\d$ cannot be defined if the exponent $p=2$ is replaced by some $p\ge \frac2{1-s}$, as highlighted by Remark 3.13 of \cite{kreisbeck_non-constant_2024}. For larger choices of $p$, $\Psi^s_\d: N^{s,p,\d}(\O)\to \R\times L^p(\Gamma_\d)$ fails to be surjective.
        
        \item A function $u\in Z^{s,\d}(\O)$ is in $\mathring H^{s,\d}(\O)$ as long as it is in either $\cA^{s,\d}_{q,\e}$ or if $u/d_\O^s\in L^2(\O_\d)$. These are both ways of quantifying the same idea that some additional integrability is required in a neighborhood of the boundary. This suggests that the deciding factor for whether $\mathring H^{s,p,\d}(\O)$ coincides with $Z^{s,p,\d}(\O)$ is how much integrability elements of $Z^{s,p,\d}(\O)$ have. Alternatively, if one wishes to work in a space where the zero-extension is continuous, it apparently suffices to just assume that the functions have some additional decay near $\partial \O$.

        \item Similar to the last point, Proposition~\ref{prop: notDense} suggests that the difference between $\mathring H^{s,\d}(\O)$ and $Z^{s,\d}(\O)$ is that the latter contains some nonzero elements of $N^{s,\d}(\O)$. Proposition 3.12 of \cite{kreisbeck_non-constant_2024} notes that $N^{s,p,\d}(\O)$ only shrinks as $p$ increases and leaves open the possibility that $N^{s,p,\d}(\O)\cap Z^{s,p,\d}(\O)$ may be trivial for sufficiently large $p$.
    \end{enumerate}
\end{remark}

Up to this point, we have only considered $s\in(0,1)$. However, the integral operator $\grad^s_\d$ defined in \eqref{eqt: Dsd} still makes sense for $s\le 0$. Thus, the norm \eqref{eqt: Hsdnorm} is well-defined and we may extend the definitions of $\mathring H^{s,\d}(\O)$ and $Z^{s,\d}(\O)$ in the natural way. It turns out that the distinction between $\mathring H^{s,\d}(\O)$ and $Z^{s,\d}(\O)$ vanishes for $s\le0$. Applications of operators given by \eqref{eqt: Dsd} with $s\le0$ include nonlocal models of heat transfer \cite{AndreuVaillo_Mazon_Rossi_ToledoMelero_2010_nonlocal_2010_book,Bobaru_Duangpanya_2010_IJHMT,Bobaru_Duangpanya_2012_JCP,chen_selecting_2015} and nonlocal elastodynamics \cite{Du2011b,EMMRICH_ETIENNE_WECKNER_OLAF_2007_CMS}, as well as nonlocal diffusion models arising in biology and materials science \cite{AndreuVaillo_Mazon_Rossi_ToledoMelero_2010_nonlocal_2010_book}.

\begin{proposition}
    Let $n$ denote a positive integer, $\O\subseteq\R^n$ a nonempty open set, and $\d>0$ a fixed parameter. If $s\le0$, then $Z^{s,\d}(\O)=\mathring H^{s,\d}(\O)$.
\end{proposition}

\begin{proof}
    As noted in \cite{cueto_variational_2023}, $\grad^s_\d: L^2(\O_\d)\to L^2(\O;\R^n)$ is a bounded linear operator when $s=0$. When $s<0$, then the kernel $\balpha^s_\d$ is integrable and so $\grad^s_\d: L^2(\O_\d)\to L^2(\O;\R^n)$ is bounded. Hence, whenever $s\le 0$, the norm given by \eqref{eqt: Hsdnorm} is equivalent to the $L^2(\O_\d)$ norm:
    \begin{align*}
        \forall u\in L^2(\O_\d), \ \norm{u}_{L^2(\O_\d)}&\le \left(\norm{u}_{L^2(\O_\d)}^2 + \norm{\grad^s_\d u}_{L^2(\O;\R^n)}^2\right)^{1/2}\\
        &=\norm{u}_{H^{s,\d}(\O)}\\
        &\le \sqrt{1+\norm{\grad^s_\d}^2}\norm{u}_{L^2(\O_\d)}.
    \end{align*}
    It follows that the closure of $\set{\phi\in C^\infty_c(\O_\d): \supp \phi\subseteq\O}$ in $H^{s,\d}(\O)$ is the same as the closure in $L^2(\O_\d)$. Therefore, when $s\le 0$,
    \begin{align*}
        \mathring H^{s,\d}(\O) = \overline{\set{\phi\in C^\infty_c(\O_\d): \supp \phi\subseteq\O}}^{L^2(\O_\d)}=\set{u\in L^2(\O_\d): u|_{\Gamma_\d}\equiv0}.
    \end{align*}
    This last set on the right-hand side is precisely $Z^{s,\d}(\O)$ when $s\le0$, so the proof is complete.
\end{proof}

\section{Conclusion}
We have further developed the theory of fractional Sobolev spaces on bounded domains, and we have highlighted the distinction between two possible choices of enforcing Dirichlet-type constraints for nonlocal problems. The space $Z^{s,\d}(\O)$ is the result of the most direct translation of zero-boundary conditions into the nonlocal setting. However, Corollary~\ref{coro: Zsd} shows that $Z^{s,\d}(\O)$ lacks many desirable properties. Comparing Theorems~\ref{thm: char H10} and \ref{thm: char Hsd0} shows that $\mathring H^{s,\d}(\O)$ is the closest analog to $H^1_0(\O)$ in terms of structural properties of the space. To establish these results, we have introduced the distributional nonlocal normal $\NN$. Using properties of the operator $\NN$, we have extended existing nonlocal integration-by-parts identities and identified when the extension-by-zero operator is bounded. 

\appendix

\section{\texorpdfstring{
A Proof of Proposition~\ref{prop: distIntegrable}
}{
Proof of the Distance-Integrability Proposition
}}
\label{app: distIntegrable}
For the sake of the reader, we include a proof of Proposition~\ref{prop: distIntegrable}.

\begin{proof}
    The proof uses the method of local coordinates, which is described in \cite{kufner_weighted_1985}. We follow their notation. We focus on the case where $\kappa\in(-1,0)$, since there is no singularity for $\kappa\ge0$.

    The assumptions on $\O$ ensure that there exist a finite number of coordinate systems
    \begin{align*}
        (\bx_i',x_{i,n}), \ \bx_i'=(x_{i,1}, x_{i,2}, \dots, x_{i,n-1}), \ \text{ for } i=1,\dots,m
    \end{align*}
    and Lipschitz continuous functions $a_i:\overline{\Delta_i}\to\R^n$, where
    \begin{align*}
        \Delta_i=\set{\bx_i': \abs{x_{i,j}}<\e_i, \ \text{ for } \ j=1,2,\dots,n-1}, \ \text{ for } i=1,\dots,m
    \end{align*}
    such that every point $\bx\in\partial\O$ has at least one $k\in\set{1,\dots,m}$ with $x_{k,n}=a_k(\bx_k')$. Additionally, there is a positive number $\beta<1$ such that the sets
    \begin{align*}
        B_i\coloneqq \set{(\bx_i',x_{i,n}): \bx_i'\in\Delta_i, a_i(\bx_i')-\beta<x_{i,n}<a_i(\bx_i')+\beta}
    \end{align*}
    satisfy
    \begin{align*}
        & U_i = B_i\cap \O = \set{(\bx_i',x_{i,n}): \bx_i'\in\Delta_i, a_i(\bx_i')-\beta<x_{i,n}<a_i(\bx_i')}\\ \text{ and }\quad&
        \Gamma_i= B_i\cap \partial\O = \set{(\bx_i',x_{i,n}): \bx_i'\in\Delta_i, x_{i,n}=a_i(\bx_i')},
    \end{align*}
    for each $i=1,\dots,m$. Then there is an open set $U_0\Subset \O$ such that the collection $\set{U_0,\dots,U_m}$ forms an open cover of $\O$.

    This provides a decomposition of the integral to be bounded:
    \begin{align*}
        \int_\O d_\O(\bx)^{\kappa}\,d\bx &\le \sum_{i=0}^m \int_{U_i} d_\O(\bx)^{\kappa}\,d\bx.
    \end{align*}
    Since $U_0$ is a positive distance away from $\partial \O$, $d_\O$ is uniformly bounded  on $U_0$. As $\O$ has finite measure, we can conclude that $\int_{U_0} d_\O^{\kappa}\,d\bx<\infty$. Hence, we only need to show that the integral over each $U_i$ is finite for $i=1,\dots,m$.

    Fix $i\in\set{1,\dots,m}$. Then Corollary 4.8 of \cite{kufner_weighted_1985} states that the following estimate holds for every $\bx=(\bx_i',x_{i,n})$ in $U_i$
    \begin{align*}
        \frac{a_i(\bx_i')-x_{i,n}}{1+L_i}\le d_\O(\bx), \ \ \ \text{ and so } \ \ \ d_\O(\bx)^\kappa\le \left(\frac{a_i(\bx_i')-x_{i,n}}{a+L_i}\right)^\kappa
    \end{align*}
    where $L_i$ the Lipschitz constant of the function $a_i:\Delta_i\to\R^n$ and $\kappa\in(-1,0)$. We may therefore use the Fubini-Tonelli theorem to write
    \begin{align*}
        \int_{U_i} d_\O^{\kappa}(\bx)\,d\bx 
        &\le \left(\frac1{1+L_i}\right)^\kappa\int_{U_i}\left( a_i(\bx_i') - x_{i,n}\right)^{\kappa}\,d\bx\\
        &=\left(\frac1{1+L_i}\right)^\kappa\int_{\Delta_i}\int_{a_i(\bx_i')-\beta}^{a_i(\bx_i')}\left(a_i(\bx_i') - x_{i,n}\right)^{\kappa}\,dx_{i,n}\,d\bx_i'\\
        &\le \left(\frac1{1+L_i}\right)^\kappa\int_{\Delta_i}\int_0^\beta r^{\kappa}\,dr\,d\bx_i',\\
        &=\frac{\beta^{\kappa+1}\abs{\Delta_i}}{(\kappa+1)(1+L_i)^\kappa}.
    \end{align*}
    where we performed the change of variables $r = a_i(\bx_i') - x_{i,n}$ inside the inner integral and used the fact that $\kappa>-1$.
\end{proof}

\section{Additional Lemmas}
This approximation lemma is used in the proof of Proposition~\ref{prop: characterize N}.
\begin{lemma}\label{lem: uniform Lq' bd} 
    Assume \ref{H1} and that $q\in\left(\frac{2}{1+s},2\right]$. Let $\bbf\in C^\infty_c(\mathop{\mathrm{Int}}\Gamma_\d;\R^n)$ be given. Then there is a sequence $\set{\bg_k}_{k=1}^\infty$ satisfying the following properties: for each $k\in\mathbb{N}$,
    \begin{enumerate}
        \item $\bg_k\in C^\infty_c(\R^n;\R^n)$.
        \item $\norm{\bell^s_\d\bbf - \bg_k}_{H^{s,\d}(\O;\R^n)} < 1/k$.
        \item There exists $C>0$ depending on $\O, q$, and $\d$ but independent of $\bbf$ and $k$ such that $\norm{\bg_k}_{L^{q'}(\Gamma_\d;\R^n)}\le C\norm{\bbf}_{L^{q'}(\Gamma_\d;\R^n)}$.
    \end{enumerate}
\end{lemma}
\begin{proof}
    Since $q\in \left(\frac{2}{1+s},2\right]$, its conjugate $q'$ satisfies $2\le q'<\frac{2}{1-s}$. Recalling Remark~\ref{rmk: lift}, for $p\le q'$, the following holds:
\begin{align}\label{eqt: bd on lift}
    \norm{\bell^s_\d\bbf}_{H^{s,p,\d}(\O;\R^n)}
    \le C\norm{\bbf}_{L^p(\Gamma_\d;\R^n)}<\infty.
\end{align}
As noted in Remark~\ref{rmk: density in Hsd}, there is therefore a sequence $\set{\bg_k}_{k=1}^\infty \subseteq C^\infty_c(\R^n;\R^n)$ such that, the restrictions $\bg_k|_{\O_\d}$ converge to $\bell^s_\d\bbf$ in $H^{s,\d}(\O)$. A construction of such an approximating sequence is produced in the Appendix of \cite{cueto_variational_2023}; we recall some of the key steps to obtain a uniform bound on $\norm{\bg_k}_{L^{q'}(\Gamma_\d;\R^n)}$.

Using the fact that $\O$ has $C^{1,1}$ boundary, we can select a partition of unity $\rho_0,\dots,\rho_{N+1}\in C^\infty_c(\R^n)$ and translation vectors $\bzeta_1,\dots,\bzeta_N\in\R^n$ such that
\begin{align*}
    \sum_{i=0}^{N+1}\rho_i = 1 \text{ in }\O_\d,
    \qquad\supp(\rho_0)\Subset\O,
    \qquad\supp(\rho_{N+1})\Subset\O^c,
\end{align*}
and, for all $\l>0$ small enough,
\begin{align*}
    \left(\supp(\rho_i)\cap\O^c\right)+\l\bzeta_i\Subset\O^c,\quad
    \text{ for each }i=1,\dots,N.
\end{align*}
Note that the choice of the functions $\rho_i$ and vectors $\bzeta_i$ are independent of $k$. Given a vector $\bz\in\R^n$, let $\tau_\bz$ denote the translation operator $\tau_\bz\bv(\bx)\coloneqq \bv(\bx-\bz)$. In \cite{cueto_variational_2023}, it is shown that for each $k\in\mathbb{N}$, there is some $\l_k>0$, such that, in addition to the property above, the function
\begin{align*}
    \bh_k\coloneqq \rho_0\bell^s_\d\bbf + \rho_{N+1}\bell^s_\d\bbf+\sum_{i=1}^N \tau_{\l_k\bzeta_i}(\rho_i\bell^s_\d\bbf)
\end{align*}
belongs to $H^{s,\d}(\O_{\d_k})$ for some $\d_k>0$ (note: $\O_{\d_k}\Supset\O$). Here, we identify $\bell^s_\d\bf$ and $\bh_k$ with there extensions by zero. Moreover, the sequence $\set{\lambda_k}_{k=1}^\infty$ can be selected so that $\set{\bh_k}_{k=1}^\infty$ satisfies
\begin{align*}
    \norm{\bell^s_\d\bbf - \bh_k}_{H^{s,\d}(\O;\R^n)} < \frac1{2k},
    \quad\text{ for all }k\in\mathbb{N}.
\end{align*}
Put $\bg_k\coloneqq \eta_k\ast \bh_k$, where each $\eta_k$ is a rescaled standard mollifier with mollifying radius smaller than $\dist(\partial\O,\partial\O_{\d_k})$ and $\norm{\bh_k - \bg_k}_{H^{s,\d}(\O;\R^n)}<\frac1{2k}$. We find $\bg_k\in C_c^\infty(\R^n;\R^n)$ and $\norm{\bell^s_\d\bbf-\bg_k}_{H^{s,\d}(\O;\R^n)} < 1/k$.

Now we estimate the $L^{q'}$ norm of $\bg_k|_{\Gamma_\d}$. Young's inequality for convolutions yields
\[
    \norm{\bg_k}_{L^{q'}(\Gamma_\d;\R^n)}
    \le \norm{\eta_k}_{L^1(\R^n)}\norm{\bh_k}_{L^{q'}(\R^n;\R^n)}
    =\norm{\bh_k}_{L^{q'}(\R^n;\R^n)}.
\]
Meanwhile, the triangle inequality and H\"older's inequality provides
\begin{multline*}
    \norm{\bh_k}_{L^{q'}(\R^n;\R^n)}
    \le\norm{\rho_0}_{L^q(\R^n)}\norm{\bell^s_\d\bbf}_{L^{q'}(\O_\d;\R^n)}
    +\norm{\rho_{N+1}}_{L^q(\R^n)}
        \norm{\bell^s_\d\bbf}_{L^{q'}(\O_\d;\R^n)}\\
    +\sum_{i=1}^N\norm{\tau_{\l_k\bzeta_i}}
        \norm{\rho_i}_{L^q(\R^n)}\norm{\bell^s_\d\bbf}_{L^{q'}(\O_\d;\R^n)}.
\end{multline*} The operator norm of the translation $\tau_{\l_k\bzeta_i}$ is bounded independently of $k$ and $i$, and the $L^q$-norms of $\rho_i$ are independent $k$. Hence, there exists $C>0$ such that $\norm{\bg_k}_{L^{q'}(\Gamma_\d;\R^n)}\le C\norm{\bell^s_\d\bbf}_{L^{q'}(\Gamma_\d;\R^n)}$, for all $k\in\mathbb{N}$. The desired result now follows from \eqref{eqt: bd on lift} with $p=q'$.

\end{proof}

The next approximation tool, Lemma~\ref{lem: approx w/cpt supp}, is used in the proof of Proposition~\ref{prop: hardy implies greens}. The idea behind the lemma, as well as the strategy of proof, are directly inspired by \cite{dyda_density_2022}.

\begin{lemma}\label{lem: approx w/cpt supp}
    Suppose \ref{H1} holds. If $u\in Z^{s,\d}(\O)$, $p\in(1,2)$, and $u/d_\O^s\in L^2(\O_\d)$, then there is a sequence $\set{u_k}_{k=1}^\infty \subseteq Z^{s,\d}(\O)$ satisfying the following three properties.
    \begin{enumerate}
        \item For each $k$, $\supp(u_k)\subseteq \O_{-1/k}$.
        \item $k^s\norm{u-u_k}_{L^2(\O)}\to0$, as $k\to\infty$. \label{property: ks}
        \item $\norm{\grad^s_\d(u-u_k)}_{L^p(\O;\R^n)}\to0$, as $k\to\infty$. \label{property: Dto0}
    \end{enumerate}
\end{lemma}
\begin{proof}
For each $k\in\N$, define $\psi_k\in C_c(\O)$ by
\[
    \psi_k(\bx) = \begin{cases}
        0, & d_\O(\bx)\le \frac1k\\
        kd_\O(\bx)-1, & \frac1k< d_\O(\bx)\le \frac2k\\
        1, & d_\O(\bx)>\frac2k.
    \end{cases}
\]
Where convenient, we identify $\psi_k$ with its extension by zero to $\R^n$. For each $k\in\N$, put $u_k:=u\psi_k$.
Since $u\in H^{s,\delta}(\O)$ and each $\psi_k$ is Lipschitz, Lemmas 3.2 and 3.3 of \cite{BeCuMC22b} ensure the nonlocal product rule holds. Thus, for each $k\in\N$, we find $u_k\in Z^{s,\d}(\O)$, with support in $\O_{-1/k}$. It remains to verify that (\ref{property: ks}) and (\ref{property: Dto0}) are valid for $\{u_k\}_{k=1}^\infty$.

For (\ref{property: ks}), since $u=u_k$ in $\O_{-2/k}$, we have
\[
    \norm{u-u_k}_{L^2(\O)}^2
    =\int_{\Gamma_{-2/k}}(1-\psi_k)^2\abs{u}^2d\bx
    \le\int_{\Gamma_{-2/k}}\abs{u}^2d\bx.
\]
Noting that $d_\O^{2s}(\bx)\le 2^{2s}k^{-2s}$ for all $\bx\in\Gamma_{-2/k}$, we may continue with
\[
    \norm{u-u_k}_{L^2(\O)}^2
    \le \int_{\Gamma_{-2/k}}\frac{\abs{u}^2}{d_\O^{2s}}d_\O^{2s}\,d\bx\le 2^{2s}k^{-2s}\norm{u/d_\O^s}_{L^2(\Gamma_{-2/k})}^2.
\]
Hence, $k^s\norm{u-u_k}_{L^2(\O)}\le 2^s\|u/d_\O^s\|_{L^2(\Gamma_{-2/k})}$.
Since $|\Gamma_{-2/k}|\to|\partial\O|=0$, as $k\to\infty$, property~\ref{property: ks}. follows from the absolute continuity of the integral.

Now, we turn to~\ref{property: Dto0}~.. From the nonlocal product rule given by Lemma 3.2 and 3.3 of \cite{BeCuMC22b} or Lemma 3.8 of \cite{bellido_nonlocal_2025}, we have
\begin{align*}
    \grad^s_\d \left(u-u_k\right) = (1-\psi_k)\grad^s_\d u - \underbrace{\int_{B_\d(\bx)}\frac{\psi_k(\bx)-\psi_k(\by)}{\abs{\bx-\by}^{n+s}}u(\by)\frac{\bx-\by}{\abs{\bx-\by}}w_\d(\bx-\by)\,d\by}_{\eqqcolon K_{\psi_k}(u)}.
\end{align*}
As in the proof of~\ref{property: ks}, since $0\le\xi_k=1-\psi_k\le1$ and $\supp\left(1-\psi_k\right)\subseteq\Gamma_{-2/k}$, we conclude that $\norm{(1-\psi_k)\grad^s_\d u}_{L^p(\O;\R^n)}\to0$. Thus, to establish property~\ref{property: Dto0}., it suffices to bound the integral term $K_{\psi_k}(u)$. 

To this end, we first perform the same change of variables and Minkowski inequality.
Then, decomposing the ball $B_\d(0)$ into $B_{1/k}(0)$ and $B_\d(0)\setminus B_{1/k}(0)$ yields
\begin{align*}
&\quad
    \norm{K_{\psi_k}(u)}_{L^p(\O;\R^n)} 
    \\&\le
    \int_{B_{1/k}(0)}\left(\int_\O \abs{\frac{\psi_k(\bx)-\psi_k(\bx-\bh)}{\abs{\bh}^{n+s}}u(\bx-\bh)w_\d(\bh)}^p\,d\bx\right)^{1/p}d\bh 
    \\&\qquad
    +\int_{B_{\d}(0)\setminus B_{1/k}(0)}\left(\int_\O \abs{\frac{\psi_k(\bx)-\psi_k(\bx-\bh)}{\abs{\bh}^{n+s}}u(\bx-\bh)w_\d(\bh)}^p\,d\bx\right)^{1/p}d\bh
    \\&
    \eqqcolon \mathsf{I}_k + \mathsf{II}_k
\end{align*}
We shall separately bound $\mathsf{I}_k$ and $\mathsf{II}_k$.

For the moment, consider a fixed $\bh\in B_{1/k}(\mathbf0)$. If $\bx\in\O_{-3/k}$, then $\bx-\bh\in\O_{-2/k}$ and so $\psi_k(\bx)=1=\psi_k(\bx-\bh)$. Thus, the only $\bx\in\O$ where the integrand could potentially be nonzero lie in $\Gamma_{-3/k}$ and so $\bx-\bh\in\Gamma_{-4/k}$. Additionally, $\abs{\psi_k(\bx)-\psi_k(\bx-\bh)}\le k\abs{\bh}$. Together, these facts allow us to reduce the domain of integration and cancel one power of $\abs{\bh}$ as follows
\begin{align*}
    \mathsf{I}_k
    &\le\int_{B_{1/k}(\mathbf{0})}\left(\int_{\Gamma_{-4/k}}
        \abs{\frac{k\abs{\bh}}{\abs{\bh}^{n+s}}
        u(\bx-\bh)w_\d(\bh)}^p\,d\bx\right)^{1/p}d\bh\\
    &\le a_0 k\int_{B_{1/k}(\mathbf{0})}\abs{\bh}^{1-n-s}
        \left(\int_{\Gamma_{-5/k}}\abs{u(\by)}^pd\by\right)^{1/p}
        d\bh.
\end{align*}
Note that we also used that fact $w_\d \le a_0$.
Compute
\begin{align*}
    \int_{B_{1/k}(\mathbf{0})}\abs{\bh}^{1-n-s}\,d\bh = \sigma_{n-1}\int_0^{1/k} r^{1-n-s}r^{n-1}\,dr = \frac{\s_{n-1}k^{s-1}}{1-s},
\end{align*}
where $\s_{n-1}$ is the surface area of the unit sphere in $\R^n$.
Since $p<2$, we combine the above work with H\"older's inequality and find
\[
    \mathsf{I}_k 
    \le Ca_0  k^s\norm{u}_{L^p(\Gamma_{-5/k})}
    \le Ck^s\abs{\Gamma_{-5/k}}^{\frac{2-p}{2p}}\norm{u}_{L^2(\Gamma_{-5/k})},
\]
where, $C<\infty$
does not depend on $k$ and may change from line to line.. Multiplying and dividing by $d_\O^{2s}$ and using the bound $d_\O(\by)^{2s}\le (5/k)^{2s}$ for $\by\in \Gamma_{-5/k}$, we obtain
\[
    \mathsf{I}_k
    \le Ck^s\abs{\Gamma_{-5/k}}^\frac{2-p}{2p}\cdot  5^{s}k^{-s}
        \left(\int_{\Gamma_{-5/k}}
            \frac{\abs{u}^2}{d_\O^{2s}}\,d\by
        \right)^{1/2}
    = C\abs{\Gamma_{-5/k}}^\frac{2-p}{2p}
        \norm{u/d_\O^s}_{L^2(\Gamma_{-5/k})}.
\]
Invoking the assumption that $u/d_\O^s\in L^2(\O_\d)$, the upper bound above must vanish as $k\to\infty$.

Now, we turn to $\mathsf{II}_k$. For convenience, put $\xi_k:=1-\psi_k$. Observe that $0\le \xi_k\le 1$, $\{\bz\in\O:\xi_k(\bz)>0\}=\Gamma_{-2/k}$, and for each $\bx,\by\in\O$,
\[
    |\psi_k(\bx)-\psi_k(\by)|=|\xi_k(\bx)-\xi_k(\by)|
    \le|\xi_k(\bx)|+|\xi_k(\by|.
\]
Focusing on the inner integral in the definition of $\mathsf{II}_k$, we may write
\begin{align}
\nonumber
    & \int_\O \abs{\frac{\xi_k(\bx)-\xi_k(\bx-\bh)}{\abs{\bh}^{n+s}}
        u(\bx-\bh)w_\d(\bh)}^p\,d\bx\\
\nonumber
    &\qqquad
    \le 2^{p-1}\frac{w_\d(\bh)^p}{|\bh|^{(n+s)p}}\left(
        \int_\O \left(|\xi_k(\bx)|^p+|\xi_k(\bx-\bh)|^p\right)|u(\bx-\bh)|^p
        \,d\bx\right)\\
\nonumber
    &\qqquad
    \le C\frac{w_\d(\bh)^p}{|\bh|^{(n+s)p}}\left(
        \int_\O\xi_k(\bx)|u(\bx-\bh)|^p\,d\bx
        +\int_\O\xi_k(\bx-\bh)|u(\bx-\bh)|^p\,d\bx\right)\\
\label{eq: B.2-II_inner}
    &\qqquad
    = C\frac{w_\d(\bh)^p}{|\bh|^{(n+s)p}}\left(
        \int_{\Gamma_{-2/k}}|u(\bx-\bh)|^p\,d\bx
        +\int_{A_k(\bh)}|u(\bx-\bh)|^p\,d\bx\right),
\end{align}
where for each $\bh\in\R^n$,
\[
    A_k(\bh):=\O\cap(\Gamma_{-2/k}+\bh)
    =\O\cap\{\bz\in\R^n:\bz-\bh\in\Gamma_{-2/k}\}.
\]

Before continuing,
we verify that $u/d_\O^{s+\e}\in L^p(\O)$, with $\e=\frac{2-p}{4p}>0$. Indeed, by H\"older's inequality,
\begin{align*}
    \int_\O\left|\frac{u(\by)}{d_\O(\by)^{s+\e}}\right|^pd\by
    \le&
    \left(\int_\O\left|
        \frac{u(\by)}{d_\O(\by)^s}\right|^2\,d\by\right)^\frac{p}{2}
    \left(\int_\O\left(
        \frac{1}{d_\O(\by)^{\e p}}\right)^{\frac{2}{2-p}}\,d\by
    \right)^\frac{2-p}{2}\\
    =&
    \|u/d_\O^s\|_{L^2(\O)}^\frac p2
    \left(\int_\O\frac{1}{d_\O(\by)^\frac12}\,d\by\right)^\frac{2-p}{2}.
\end{align*}
Proposition~\ref{prop: distIntegrable} implies $1/d_\O^\frac{1}{2}\in L^1(\O)$, so the upper bound above is finite.

Now, we bound the two integrals in~\eqref{eq: B.2-II_inner}. If either $\bx\in\Gamma_{-2/k}$ or $\bx\in A_k(\bh)$, then
\[
    d_{\O}(\bx-\bh)\le\frac2k+|\bh|\le3|\bh|,
\]
whenever $|\bh|\ge1/k$. It follows that for $|\bh|\ge 1/k$,
\begin{align*}
    &\int_{\Gamma_{-2/k}}|u(\bx-\bh)|^p\,d\bx
        +\int_{A_k(\bh)}|u(\bx-\bh)|^p\,d\bx\\
    &\qquad
    =\int_{\Gamma_{-2/k}}d_\O(\bx-\bh)^{(s+\e)p}\left|
            \frac{u(\bx-\bh)}{d_\O(\bx-\bh)^{s+\e}}\right|^pd\bx\\
    &\qqqquad\qqquad
        +\int_{A_k(\bh)}d_\O(\bx-\bh)^{(s+\e)p}\left|
            \frac{u(\bx-\bh)}{d_\O(\bx-\bh)^{s+\e}}\right|^pd\bx\\
    &\qquad
    \le(3|\bh|)^{(s+\e)p}\left(
        \int_{\Gamma_{-2/k}}\left|
            \frac{u(\bx-\bh)}{d_\O(\bx-\bh)^{s+\e}}\right|^pd\bx
        +\int_{A_k(\bh)}\left|
            \frac{u(\bx-\bh)}{d_\O(\bx-\bh)^{s+e}}\right|^pd\bx
    \right).
\end{align*}
Since $\supp u\subseteq\ov{\O}$ and $\bx-\bh\in A_k(\bh)\Rightarrow\bx\in\Gamma_{-2/k}$, by making the change of variables $\by\leftrightarrow\bx-\bh$, we obtain
\begin{multline}\label{eq: B.2-inner bound}
    \int_{\Gamma_{-2/k}}|u(\bx-\bh)|^p\,d\bx
        +\int_{A_k(\bh)}|u(\bx-\bh)|^p\,d\bx\\
    \le(3|\bh|)^{(s+\e)p}\left(
        \int_{A_k(-\bh)}\left|
            \frac{u(\by)}{d_\O(\by)^{s+\e}}\right|^p\,d\by
        +\int_{\Gamma_{-2/k}}\left|
            \frac{u(\by)}{d_\O(\by)^{s+\e}}\right|^p\,d\bx
    \right).
\end{multline}

Returning to~\eqref{eq: B.2-inner bound}, we will argue that the two integrals vanish as $k\to\infty$. Let $\eta>0$ be given. Recall that $A_k(\bh)$ is contained in a translation of $\Gamma_{-2/k}$. Thus, for any $\bh\in\R^n$, we have $\sup_{\bh\in\R^n}|A_k(-\bh)|\le |\Gamma_{-2/k}|$ and, moreover, that $|\Gamma_{-2/k}|\to 0$, as $k\to\infty$. Thus, absolute continuity of the integral implies there is a $k_\eta\in\mathbb{N}$ such that for all $k\ge k_\eta$,
\[
    \sup_{\bh\in B_\d(0)\setminus B_{1/k}(0)}\int_{A_k(-\bh)}\left|
        \frac{u(\by)}{d_\O(\by)^{s+\e}}\right|^p\,d\bx
    +\int_{\Gamma_{-2/k}}\left|
        \frac{u(\by)}{d_\O(\by)^{s+\e}}\right|^p\,d\bx
    <\eta.
\]
Collecting this with the bounds in~\eqref{eq: B.2-inner bound} and~\eqref{eq: B.2-II_inner}, for all $\bh\in B_\d(\mathbf{0})\setminus B_{1/k}(\mathbf{0})$ and $k\ge k_\eta$, we have
\begin{align*}
    \int_\O \abs{\frac{\xi_k(\bx)-\xi_k(\bx-\bh)}{\abs{\bh}^{n+s}}
        u(\bx-\bh)w_\d(\bh)}^p\,d\bx
    \le&
    C\eta^p|\bh|^{-(n+s)p}|\bh|^{(s+\e)p}w_\d(\bh)^p\\
    \le&
    C\eta^p|\bh|^{-(n-\e)p},
\end{align*}
where $C<\infty$ is a constant that is independent of $k$ and $|\bh|$.
It follows that
\[
    \mathsf{II}_k
    \le C\eta\int_{B_\d(\mathbf{0})\setminus B_{1/k}(\mathbf{0})}\frac{1}{|\bh|^{n-\e}}\,d\bh
    \le C\frac{\eta\delta^\e}{\e}.
\]
As $\eta>0$ was arbitrary, we conclude that $\lim_{k\to\infty}\mathsf{II}_k=0$.
\end{proof}

\section{\texorpdfstring{
Comments on Example~\ref{ex: ustar}
}{
Comments on the Numerical Example
}}
\label{app: numerics}
To find a function $u_*$ described in Corollary~\ref{coro: Zsd}, it suffices to find a nonzero function in $Z^{s,\d}(\O)$ with zero nonlocal gradient. As noted in \cite{kreisbeck_non-constant_2024}, one can instead find a function $u_*$ such that $Q^s_\d\ast u_*$ is constant in $\O$, where $Q^s_\d$ is the nonlocal-to-local translation operator defined in \cite{bellido_non-local_2023}. Thus, once we have a discrete, matrix-representation of the convolution operator $u\mapsto Q^s_\d\ast u$, the problem is reduced to solving a linear system.

We discretized in terms of the Chebyshev points on the interval $[-1,1]$, which have a number of desirable numerical properties \cite{trefethen_approximation_2019}. However, some care is needed to construct a faithful depiction of $u_*$: the function $Q^s_\d$ is defined in terms of integrating a strongly singular function very close to its singularity, the resulting function $Q^s_\d$ is itself weakly singular, and the solution $u_*$ has singularities at the endpoints. At each stage of the computation, we approach the limits of what can be represented in floating point arithmetic.

We briefly describe the general approach to navigate these difficulties. First, we exploit the fact that $w_\d$ is constant on a neighborhood of the origin. This means that the behavior of $Q^s_\d$ near zero can be worked out analytically, on paper. This removes the most dangerous points: $Q^s_\d(z)$ is hardest to evaluate numerically when $z\sim0$, since the inputs approach the strong singularity in the integral defining $Q^s_\d$. For points outside the neighborhood where $Q^s_\d$ is known analytically, we compute the defining integral using an adaptive double-exponential quadrature scheme. This is a good general purpose numerical integrator which excels at handling endpoint singularities. The adaptive quadrature terminates when two consecutive estimates $E_k, E_{k+1}$ satisfy
\begin{align*}
    \abs{E_{k+1}-E_k} < \abs{E_{k+1}}\cdot\mathsf{relTol} + \mathsf{absTol}, \ \text{ where } \mathsf{relTol} = 1\times10^{-14} = \mathsf{absTol}.
\end{align*}
Together, the analytic formula for $Q^s_\d$ near the origin and the numerical integration scheme provide a high-precision method for evaluating $Q^s_\d$ at arbitrary points in its domain. Although each point evaluation of $Q^s_\d$ requires an additional integration, the process can be easily accelerated via vectorization.

To compute the convolution $Q^s_\d\ast u$, we further exploit the fact that $Q^s_\d$ behaves exactly like a known power function at the origin. This means that for each $x\in\O$, $(Q^s_\d\ast u)(x)$ can be written as a sum of integrals, each with a fixed power-function singularity at one endpoint. As a result, we can apply Gauss-Jacobi quadrature, using Chebyshev interpolation to approximate $u$ with a polynomial. Since the number of Chebyshev points, $N$, is fixed, there is an upper bound on the degree of the Chebyshev series representing $u$. Hence, we can choose a quadrature scheme of high enough order to integrate any such polynomial \textit{exactly}, with no quadrature error; of course, there is still the discretization and interpolation error in representing $u$ on a finite grid, as well as the rounding error inherent to floating point arithmetic.

This provides a method of evaluating $Q^s_\d\ast u$, which is enough to construct a matrix representation of that convolution operator. We solved the resulting linear system to find $u_*^N$, an approximation to $u_*$ using $N$ Chebyshev points. Using the Holland Computing Center's resources, we computed $u_*^N$ for a sequence of values of $N$, including $N\in\set{257,513,1025,2049,4097,4500,6145}$, and observed convergence to some fixed function $u_*$. After fitting a power function to the endpoint behavior of each $u_*^N$, we also saw convergence of the exponents and coefficients, as well as $R^2$ values approaching a perfect fit of $R^2=1$. We interpret this as compelling evidence that the method is genuinely approximating a real function in $Z^{s,\d}(\O)$ with the desired properties.

As additional verification, we computed $\grad^s_\d u_*^N$ , without using the matrix construction described above. If our method was working correctly, then the truncated-fractional gradient should be zero on $\O$. Note, however, that even just having $\grad^s_\d u_*^N\in L^2(\O_\d)$ would be sufficient for demonstrating that $Z^{s,\d}(\O)\not= \mathring H^{s,\d}(\O)$. This is because $u_*^N$ seemingly violates the fractional Hardy inequality and therefore cannot be in $\mathring H^{s,\d}(\O)$.

To compute $\grad^s_\d u_*^N$, we used MATLAB's builtin \texttt{integral} command to compute the convolution on an evenly spaced grid of 100 points. Since the function is not defined at $x=\pm1$, we used the interval $[-0.999999,0.999999]$. The MATLAB command uses a change-of-variables to mitigate the endpoint singularities, followed by globally adaptive Gauss-Kronrod quadrature. The result was that $\abs{\grad^s_\d u_*^N(x)}<1.4090\times10^{-4}$ for all 100 points tested, with a relative tolerance set to $10^{-8}$. The highest values were at the endpoints. Excluding the points $x=\pm0.999999$, we obtained $\abs{\grad^s_\d u_*^N(x)}<2.0130\times10^{-6}$. Some possible explanations for this behavior include that the function is not fully resolved, that the power-function fit passed to \texttt{integral} was chosen incorrectly, or that the singularities present at each stage of the computation lead to unavoidable quadrature error. Indeed, when MATLAB returned its answer, it also printed the warning ``There may be a singularity, or the tolerances may be too tight for this problem." In any case, the two independent methods both suggest that $\abs{\grad^s_\d u_*^N(x)}$ remains close to zero even near the endpoints. This provides additional evidence that $u_*^N$ is genuinely approximating a function in $Z^{s,\d}(\O)$ which is not in $\mathring H^{s,\d}(\O)$.

\section*{Acknowledgments}
\noindent
This work was completed utilizing the Holland Computing Center of the University of Nebraska, which receives support from the UNL Office of Research and Innovation, and the Nebraska Research Initiative.

\section*{Declarations}

\subsection*{Data Availability Statement}
The datasets and code generated and/or analyzed during the current study are available from the authors upon reasonable request.

\subsection*{Conflict of Interest Statement}
On behalf of authors M.~F., A.~L., and M.~P., the corresponding author states that there is no conflict of interest. 

\subsection*{Author Contribution Statement}
Authors M.~F., A.~L., and M.~P. contributed equally to the paper.

\subsection*{Funding Declaration}
 \noindent
Author M.~F. was partially supported by NSF grant DMS-2109149.
Authors A.~L. and M.~P. were partially supported by NSF grants DMS-2510494 and CMMI-1953346.
Author A.~L. was also partially supported by USGS Grant No. G23AC00156-01.

\end{document}